\documentclass[a4paper,11pt]{article}

\usepackage{amssymb,amsthm,amsmath,amscd,amsfonts,bbm,mathrsfs}
\usepackage{hyperref}
\usepackage{xcolor}

\usepackage[cmtip,all]{xy}

\theoremstyle{plain}
\newtheorem{Lemma}{Lemma}
\newtheorem{Thm}[Lemma]{Theorem}
\newtheorem{Prop}[Lemma]{Proposition}
\newtheorem{Cor}[Lemma]{Corollary}
\theoremstyle{definition}
\newtheorem{Defn}[Lemma]{Definition}

\theoremstyle{remark}
\newtheorem{Remark}[Lemma]{Remark}
\newtheorem*{Remark*}{Remark}
\newtheorem{Example}[Lemma]{Example}

\theoremstyle{definition}

\numberwithin{Lemma}{section}
\numberwithin{Subpoint}{Lemma}
\numberwithin{equation}{section}

\newcommand{\AAA}{\mathcal{A}}
\newcommand{\BBB}{\mathcal{B}}

\newcommand{\III}{\mathcal{I}}
\newcommand{\LLL}{\mathcal{L}}
\newcommand{\MMM}{\mathcal{M}}
\newcommand{\NNN}{\mathcal{N}}
\newcommand{\OOO}{\mathcal{O}}

\newcommand{\RRR}{\mathcal{R}}

\newcommand{\TTT}{\mathcal{T}}

\newcommand{\Fm}{\mathfrak{m}}

\newcommand{\FX}{\mathfrak{X}}

\newcommand{\FF}{{\mathbb{F}}}

\newcommand{\ZZ}{{\mathbb{Z}}}

\DeclareMathOperator{\BT}{BT}
\DeclareMathOperator{\Dcat}{D}
\DeclareMathOperator{\Dfun}{\mathbf D} %\DD}
\DeclareMathOperator{\DDcat}{DD}
\DeclareMathOperator{\DDfun}{\mathbf {DD}} %\DD\DD}
\DeclareMathOperator{\DFcat}{DF}
\DeclareMathOperator{\DFfun}{\mathbf{DF}} %\DD\FF}

\DeclareMathOperator{\CRIS}{CRIS}

\DeclareMathOperator{\Disp}{Disp}

\DeclareMathOperator{\Filone}{Fil^1}

\DeclareMathOperator{\Hom}{Hom}
\DeclareMathOperator{\HPD}{HPD}

\DeclareMathOperator{\Idem}{Idem}

\DeclareMathOperator{\Ker}{Ker}
\DeclareMathOperator{\LF}{LF}
\DeclareMathOperator{\Lie}{Lie}

\DeclareMathOperator{\Spec}{Spec}

\DeclareMathOperator{\Rad}{Rad}

\DeclareMathOperator{\Spf}{Spf}

\DeclareMathOperator{\Win}{Win}
\DeclareMathOperator{\Zar}{Zar}

\newcommand{\hatotimes}{\mathbin{\hat\otimes}}

\DeclareMathOperator{\cris}{cris}

\DeclareMathOperator{\id}{id}
\DeclareMathOperator{\nil}{nil}
\DeclareMathOperator{\per}{per}
\DeclareMathOperator{\red}{red}

\newcommand{\nameoftop}{{\operatorname{pr}}}

\renewcommand{\u}{\underline}

\newcommand{\chang}[1]{{#1}}

\begin{document}

\title{Divided Dieudonn\'e crystals}
\author{Eike Lau}
%\date{\today}
%\address{Fakult\"{a}t f\"{u}r Mathematik,
%Universit\"{a}t Bielefeld, D-33501 Bielefeld}
%\email{lau@math.uni-bielefeld.de}
%\subjclass[2020]{14L05, 14F30}
%\keywords{$p$-divisible groups, Dieudonn\'e crystals, displays}

\begingroup
\renewcommand{\thefootnote}{}
\footnotetext{2020 Mathematics Subject Classification. 14L05, 14F30}
\footnotetext{Key words. $p$-divisible groups, Dieudonn\'e crystals, displays}
\endgroup

\maketitle

\begin{abstract}
We define a category of divided Dieudonn\'e crystals which classifies $p$-divisible groups over schemes in characteristic $p$ with certain finiteness conditions, including all $F$-finite noetherian schemes. For formally smooth schemes or locally complete intersections this generalizes and extends known results on the classical crystalline Dieudonn\'e functor.
\end{abstract}

\setcounter{tocdepth}{1}
\tableofcontents

\section{Introduction}

\begingroup

\chang
{Let $p$ be a prime and let $X$ be an $\FF_p$-scheme. We}
consider the crystalline Dieudonn\'e functor 
from $p$-divisible groups to Dieudonn\'e crystals,
\[
\Dfun_X:\BT(X)\to\Dcat(X),
\]
which is outlined in \cite{Grothendieck:Nice, Grothendieck:Montreal}
and constructed in \cite{Mazur-Messing, BBM}.
The Hodge filtration of a $p$-divisible group 
gives a refinement of $\Dfun_X$ to a functor
\[
\DFfun_X:\BT(X)\to\DFcat(X)
\]
from $p$-divisible groups
to Dieudonn\'e crystals with an admissible filtration as defined
in \cite{Grothendieck:Montreal}, called filtered Dieudonn\'e crystals
in the following.

The functors $\Dfun_X$ and $\DFfun_X$ can be expected to have good properties 
only if $X$ is PD torsion free in the sense that
the universal PD envelopes of affine open subschemes of $X$ are torsion free.
We will define a category of divided Dieudonn\'e crystals 
which compensates the effect of PD torsion.
Here `divided' refers to the fact that 
the Frobenius operator of the crystal is divided by $p$ on the filtration module.
It is clear that the image of the filtration under the Frobenius 
is divisible by $p$, but if torsion occurs, 
the division is not unique and thus
carries additional information.

In more detail,
following \cite{Fontaine-Messing} let $\OOO_X^{\cris}$ be the presheaf of rings
on the category of $X$-schemes defined as the direct image
of the crystalline structure sheaf.
We use a topology $\nameoftop$ 
which is generated by the Zariski topology 
and by infinite successions of extractions of $p$-th roots
(Definition \ref{Def:top}).
Then $\OOO_X^{\cris}$ is a $\nameoftop$-sheaf, which carries an endomorphism $\sigma$ induced by the Frobenius of $X$ and a PD ideal $\III_X^{\cris}\subseteq\OOO_X^{\cris}$
with quotient $\OOO_X$.
The definition of divided Dieudonn\'e crystals 
is based on the following fact.

\begin{Lemma}
\label{Le:Int-sigma1}
(Lemma \ref{Le:sigma1-O})
There is a natural $\sigma$-linear map 
\[
\sigma_1:\III_X^{\cris}\to\OOO_X^{\cris}
\]
with $p\sigma_1=\sigma$.
\end{Lemma}

This means that the sheaf $\OOO_X^{\cris}$ 
carries a frame structure in the sense of \cite{Lau:Frames},
and we define divided Dieudonn\'e crystals as windows over this frame
in the following sense.

\begin{Defn}
\label{De:DDX-Intro}
A divided Dieudonn\'e crystal over $X$ is a collection 
\[
\u\MMM=(\MMM,\MMM_1,\Phi,\Phi_1)
\] 
where $\MMM$ is a finite locally free $\OOO_X^{\cris}$-module
with respect to the $\nameoftop$-topology,
$\MMM_1\subseteq \MMM$ is a submodule with $\III_X^{\cris}\MMM\subseteq\MMM_1$
such that $\MMM/\MMM_1$ is a finite locally free $\OOO_X$-module,
$\Phi:\MMM\to\MMM$ and $\Phi_1:\MMM_1\to\MMM$ 
are $\sigma$-linear maps with $\Phi_1(ax)=\sigma_1(a)\Phi(x)$ 
for local sections $a\in\III_X^{\cris}$ and $x\in\MMM$,
such that $\Phi(\MMM)+\Phi_1(\MMM_1)$ generates $\MMM$.
\end{Defn}

The category of divided Dieudonn\'e crystals will be denoted by $\DDcat(X)$.
There is a forgetful functor $\DDcat(X)\to\DFcat(X)$
from divided Dieudonn\'e crystals to filtered Dieudonn\'e crystals,
which is an equivalence if $X$ is PD torsion free;
see Proposition \ref{Pr:tf-DD-DF}.

The scheme $X$ is called $F$-finite 
if the Frobenius morphism $\phi:X\to X$ is finite, 
and $X$ \chang{will be} called $F$-nilpotent if the kernel of $\phi:\OOO_X\to\OOO_X$
is locally a nilpotent ideal. 
Every locally noetherian $\FF_p$-scheme is $F$-nilpotent.

\begin{Thm}
\label{Th:Int-FF}
For \chang{every} $\FF_p$-scheme $X$ there is a functor
\[
\DDfun_X:\BT(X)\to\DDcat(X)
\]
from $p$-divisible groups to divided Dieudonn\'e crystals.
The functor $\DDfun_X$ is an equivalence of categories
if $X$ is $F$-finite and $F$-nilpotent.
\end{Thm}

See Proposition \ref{Pr:DD-functor} and Theorem \ref{Th:Ffin-Fnil}.

\begin{Remark}
The usual description of crystals by modules with a connection
extends to a description of divided Dieudonn\'e crystals by
windows with a connection; 
see Section \ref{Se:Explicit-DD} for details,
in particular Definition \ref{Def:win-nabla} and Corollary \ref{Co:BT-Win}.
\end{Remark}

\begin{Remark}
Definition \ref{De:DDX-Intro} can be extended to schemes $X$ where $p$ is locally nilpotent,
and again there is a functor $\DDfun_X$ as above.
If $p\ge 3$ and the reduction
$X_{\FF_p}$ is $F$-finite and $F$-nilpotent,
then the functor $\DDfun_X$ is an equivalence 
by an obvious extension of Theorem \ref{Th:Int-FF}.
\end{Remark}

\subsection*{Outline of the proof}

By $\nameoftop$-descent, 
Lemma \ref{Le:Int-sigma1} and Theorem \ref{Th:Int-FF}
are reduced to the case where $X=\Spec R$ is affine and semiperfect
in the sense that the Frobenius endomorphism of $R$ is surjective.
Then $\OOO_X^{\cris}(X)=A_{\cris}(R)$ 
is the universal \chang{$p$-complete} PD thickening of $R$,
and Lemma \ref{Le:Int-sigma1} follows from 
the corresponding statement for $A_{\cris}$, 
which is proved in \cite{Scholze-Weinstein}.
Moreover, divided Dieudonn\'e crystals over $X$ are equivalent to 
windows over the frame $\u A{}_{\cris}(R)$.
Hence the results of \cite{Lau:Semiperfect} give the functor $\DDfun_X$ 
and prove the equivalence if $R$ is iso-balanced;
this is a technical condition that implies $F$-nilpotence.
The proof of the equivalence when $R$ is $F$-nilpotent 
uses similar methods, based on a deformation to the perfect case
(Section \ref{Se:semiperfect}).

\subsection*{Implications for the classical functors}

We assume again that $X$ is 
\chang{an $\FF_p$-scheme} and
refer to the main text for variants when $p$ is locally nilpotent. By a simple analysis of the forgetful functor $\DDcat(X)\to\Dcat(X)$, Theorem \ref{Th:Int-FF} implies the following.

\begin{Cor}
\label{Co:Int-isog}
(Corollary \ref{Co:Isogeny})
\chang{If $X$ is $F$-finite and $F$-nilpotent,
then the functor\/ $\Dfun_X$ is fully faithful up to isogeny.
More precisely, for $p$-divisible groups $G$, $H$ over $X$
the homomorphism}
\[
\Hom(G,H)\to\Hom(\Dfun_X(G),\Dfun_X(H))
\]
is injective with cokernel annihilated by $p$.
\end{Cor}

\chang{If $X$ is PD torsion free, the functors $\DDfun(X)$ and $\DFfun(X)$ are interchangeable, and Theorem \ref{Th:Int-FF} can be restated as follows.}

\begin{Cor}
\label{Co:Int-FFT}
(Theorem \ref{Th:eq-Ffin-Fnil-tf})
If $X$ is $F$-finite, $F$-nilpotent, and PD torsion free,
then the functor $\DFfun_X$ is an equivalence.
\end{Cor}

An l.c.i.\ scheme is a locally noetherian scheme whose
complete local rings are complete intersection rings.
\chang{Every} $F$-finite locally noetherian scheme is excellent,
and \chang{every} excellent l.c.i.\ scheme is PD torsion free.
\chang{Hence the following result is a particular case of Corollary \ref{Co:Int-FFT}.}

\begin{Cor}
\label{Co:Int-lci}
(Corollary \ref{Co:eq-lci})
If $X$ is an $F$-finite l.c.i.\ scheme, 
then the functor $\DFfun_X$ is an equivalence.
\end{Cor}

\chang{This applies in particular if $X$ is $F$-finite and regular,
or equivalently if $X$ is locally noetherian and locally has a finite $p$-basis.

More generally, if $X$ locally has a $p$-basis, then $X$ is PD torsion free, moreover the forgetful functor $\DFcat(X)\to\Dcat(X)$ is an equivalence, and hence the three functors $\DDfun_X$, $\DFfun_X$, and $\Dfun_X$ are interchangeable. In this case, the $F$-finiteness condition of Corollary \ref{Co:Int-FFT} can be omitted as follows.}

\begin{Thm}
\label{Th:Int-pbasis}
(Theorem \ref{Th:eq-p-basis})
If $X$ locally has a $p$-basis, then the 
functor $\Dfun_X$ is an equivalence.
\end{Thm}

This applies for example when $X$ is smooth over a field.
\chang{The proof of Theorem \ref{Th:Int-pbasis}
is again a reduction to the semiperfect case by
$\nameoftop$-descent, using that the relevant semiperfect rings
have an explicit description that allows to apply
a variant of the methods of \cite{Lau:Semiperfect}.}

\begin{Remark}
\chang{Corollary \ref{Co:Int-FFT} can be proved directly by
$\nameoftop$-descent from the semiperfect case,
without reference to divided Dieudonn\'e crystals.
This route will be followed in the main text in order to reach the results for the classical functors as directly as possible. This causes a slight duplication of effort, notably the proofs of Theorems \ref{Th:eq-Ffin-Fnil-tf} and \ref{Th:Ffin-Fnil} are essentially parallel.}
\end{Remark}

\begin{Remark}
\chang{The intersection of Theorems \ref{Th:Int-FF} and \ref{Th:Int-pbasis} consists of $\FF_p$-schemes which locally admit a finite $p$-basis. It is not clear if the finiteness hypotheses in Theorem \ref{Th:Int-FF} are necessary.}
\end{Remark}

\begin{Remark}
\chang{The forgetful functor $\DFcat(X)\to\Dcat(X)$ is fully faithful if $X$ is an l.c.i.\ scheme and an equivalence if $X$ locally has a $p$-basis.}
Hence the conclusions of  
Theorem \ref{Th:Int-pbasis} and Corollary \ref{Co:Int-lci}
mean that the functor $\Dfun_X$ is fully faithful 
with essential image $\DFcat(X)$,
which confirms \cite[p.~108, Probl\`eme 1]{Grothendieck:Montreal}
in these cases.
\end{Remark}

\begin{Remark}
A weaker version of Corollary \ref{Co:Int-lci} was announced in
\cite{Lau:OWR2015}, 
based on a relative version of the theory of Dieudonn\'e displays
of \cite{Zink:DDisp, Lau:Relations}.
\end{Remark}

\subsection*{Previously known properties of the Dieudonn\'e functors}

Corollaries \ref{Co:Int-isog}, \ref{Co:Int-FFT}, \ref{Co:Int-lci}
and Theorem \ref{Th:Int-pbasis}
extend known results on the functors $\Dfun_X$ and $\DFfun_X$,
which we try to collect here.

\begin{enumerate}
\item
If $X$ is perfect 
(the case of an empty $p$-basis),
the functor $\Dfun_X$ and its analogue
for finite group schemes are an equivalence
by a result of Gabber, proved in \cite{Berthelot:Parfait} 
for perfect valuation rings;
see also \cite{Lau:Smoothness}.
\item
If $X$ is normal and locally irreducible 
\chang{and locally has a $p$-basis,}
then $\Dfun_X$ and its analogue
for finite group schemes are fully faithful
by \cite{Berthelot-Messing}.
\item
If $\FX$ is a formally smooth formal scheme in characteristic $p$ 
such that $\FX_{\red}$ is of finite type over a field with a finite
$p$-basis, 
the functor $\Dfun_\FX$ is an equivalence by \cite{Jong:Crystalline}.
This can be deduced from the case of $F$-finite regular schemes,
which is part of Theorem \ref{Th:Int-pbasis}; see Remark \ref{Rk:dJ-case}.
\item
If $X$ is of finite type over a field with a finite $p$-basis,
the functor $\Dfun_X$ is fully faithful up to isogeny by \cite{Jong:Crystalline}.
\item
If $X$ is a locally noetherian excellent l.c.i.\ scheme,
then $\Dfun_X$ is fully faithful by \cite{Jong-Messing}.
Since $F$-finite and noetherian implies excellent,
this covers the full faithfulness part of Corollary \ref{Co:Int-lci}.
\item
If $X$ is the spectrum of a perfect ring modulo a finite regular sequence,
the functor $\Dfun_X$ is fully faithful by \cite{Scholze-Weinstein}.
This is also a consequence of Corollary \ref{Co:Int-FFT}.
\item 
If $\OOO_K$ is a complete discrete valuation ring of characteristic zero with perfect residue field of characteristic $p\ge 3$,
Corollary \ref{Co:Int-lci} for the scheme $X=\Spec\OOO_K/p$
is equivalent to Breuil's classification of $p$-divisible groups over $\OOO_K$ by filtered modules in \cite{Breuil:Groupes}.
Breuil's result was extended to a relative setting (formally smooth over $\OOO_K$) in \cite{Kim:Relative} and to complete regular local rings of higher dimension in \cite{Cais-Lau};
this corresponds to additional cases of Corollaries \ref{Co:Int-FFT} and \ref{Co:Int-lci}.
\end{enumerate}

\subsection*{\chang{Relation with recent work}}

\chang{Shortly after this article appeared online, a prismatic Dieudonn\'e theory was developed in \cite{Anschuetz-LeBras}, whose main result is an equivalence between $p$-divisible groups and prismatic Dieudonn\'e crystals over quasi-syntomic ($p$-complete) rings.} 
For $\FF_p$-algebras, prismatic and crystalline Dieudonn\'e modules coincide, moreover quasi-syntomic $\FF_p$-algebras are PD torsion free. So \cite{Anschuetz-LeBras} and Theorem \ref{Th:Int-FF} above extend crystalline Dieudonn\'e theory in different directions. The results of \cite{Anschuetz-LeBras} also yield Corollary \ref{Co:Int-FFT} and Theorem \ref{Th:Int-pbasis}.

\chang{More recently, a classification of $p$-divisible groups and their truncated analogues over general $p$-nilpotent schemes in terms of vector bundles on the syntomification was proved in \cite{Gardner-Madapusi}, building on the theory of prismatization of Bhatt--Lurie and Drinfeld.}

\bigskip
\noindent
This article is organized as follows. 
Sections \ref{Se:Fin-Frob}--\ref{Se:frames}
contain generalities on $F$-finite schemes, 
the crystalline Dieudonn\'e functor, 
schemes with torsion free PD envelopes,
and the notion of frames, including a sheaf version.
After the results of \cite{Lau:Semiperfect} are extended
to more general classes of semiperfect rings 
in section \ref{Se:semiperfect},
Corollary \ref{Co:Int-FFT} and Theorem \ref{Th:Int-pbasis}
are deduced using the $\nameoftop$-topology 
in section \ref{Se:descent}.
The notion of divided Dieudonn\'e crystals is introduced 
in section \ref{Se:divided}, and Theorem \ref{Th:Int-FF} is proved in section \ref{Se:functor}.
Section \ref{Se:Explicit-DD} gives an explicit description of divided Dieudonn\'e crystals in terms of windows over PD envelopes with a connection.
Finally, the relation between
divided Dieudonn\'e crystals and the display associated to
a $p$-divisible group as in \cite{Zink:Display, Lau:Smoothness}
is briefly discussed in section \ref{Se:displays}.

\subsection*{Acknowledgements}

\chang{The author is grateful to the anonymous referee for helpful comments. 
This work was 
funded by the Deutsche Forschungsgemeinschaft (DFG, German Research
Foundation) -- Project-ID 491392403 -- TRR 358.
}

\endgroup

\section{Finiteness conditions on the Frobenius map}
\label{Se:Fin-Frob}

An $\FF_p$-algebra $R$ is called $F$-finite if 
the Frobenius endomorphism $\phi_R:R\to R$ is finite.
An $F$-finite noetherian $\FF_p$-algebra is excellent by 
\cite[Th.~2.5]{Kunz:Noetherian}.
\chang{Following \cite{Berthelot-Messing}, a $p$-basis of an $\FF_p$-algebra $R$ is a family of elements $(x_i)_{i\in I}$ of $R$ such that the homomorphism 
\begin{equation}
\label{Eq:p-basis}
R[\{T_i\}_{i\in I}]/(\{T_i^p-x_i\}_{i\in I})\to R
\end{equation}
that extends $\phi_R$ by $T_i\mapsto x_i$ is bijective.}
For reference we recall:

\begin{Lemma}
\label{Le:F-finite-reg}
For a noetherian $F$-finite\/ $\FF_p$-algebra $R$ the following are equivalent:
\begin{enumerate}
\item
\label{Le:F-finite-reg-pbasis}
\chang{$R$ locally has a finite $p$-basis,}
\item
\label{Le:F-finite-reg-smooth}
$R$ is formally smooth over $\FF_p$,
\item
\label{Le:F-finite-reg-top-smooth}
$R$ is formally smooth over $\FF_p$ with respect to the $I$-adic topology
for an ideal $I\subseteq\Rad(A)$,
where $\Rad(R)$ is the Jacobson radical,
\item
\label{Le:F-finite-reg-regular}
$R$ is regular.
\end{enumerate}
\end{Lemma}

\begin{proof}
\eqref{Le:F-finite-reg-pbasis}$\Rightarrow$\eqref{Le:F-finite-reg-smooth} 
follows from \cite[$0_{\,\rm IV}$ (21.2.7)]{EGA}
or \cite[Lemma 1.1.2]{Jong:Crystalline}, 
and \eqref{Le:F-finite-reg-smooth}$\Rightarrow$\eqref{Le:F-finite-reg-top-smooth} 
is clear.
\eqref{Le:F-finite-reg-top-smooth}$\Rightarrow$\eqref{Le:F-finite-reg-regular}
follows from \cite[$0_{\,\rm IV}$ (22.5.8)]{EGA} or \cite[Thm.~28.7]{Matsumura:Ring}
applied to the localizations of $A$ at all maximal ideals.
To prove \eqref{Le:F-finite-reg-regular}$\Rightarrow$\eqref{Le:F-finite-reg-pbasis}
one can assume that $R$ is local.
Let $a_1,\ldots,a_r\in R$ map to a $p$-basis of the residue field $R/\Fm_R$
and let $a_{r+1},\ldots,a_n$ be a minimal set of generators of the maximal
ideal $\Fm_R$. 
Since the completion $\hat R$ is isomorphic to a power series ring,
$a_1,\ldots,a_n$ form a $p$-basis of $\hat R$, 
and thus a $p$-basis of $R$ by faithfully flat descent;
note that $\hat R\otimes_{R,\phi}R\cong\hat R$ via $\phi\otimes 1$.
\end{proof}

\begin{Remark}
\label{Rk:Rings-dJ}
The basic rings in \cite[1.3.1]{Jong:Crystalline} are noetherian 
$\FF_p$-algebras $A$ which are $I$-adically complete and formally
smooth for the $I$-adic topology such that $A/I$ is of finite type
over a field $k$ with a finite $p$-basis. Such rings are $F$-finite
and thus satisfy the equivalent conditions of Lemma \ref{Le:F-finite-reg}.
\end{Remark}

\begin{Remark}
\label{Re:p-basis-formally-smooth}
\chang{The implication \eqref{Le:F-finite-reg-pbasis}$\Rightarrow$\eqref{Le:F-finite-reg-smooth} of Lemma \ref{Le:F-finite-reg} generalises as follows:
For any $\FF_p$-algebra $R$ with a $p$-basis $(x_i)_{i\in I}$, the homomorphism $\FF_p\to R$ satisfies the lifting criterion of formal smoothness with respect to thickenings of $\FF_p$-algebras $\pi:B\to B/J$ with $\phi_B(J)=0$. This follows from the proof of \cite[$0_{\,\rm IV}$ (21.2.7)]{EGA}. Explicitly, let $\bar\phi:B/J\to B$ be the homomorphism induced by $\phi_B$.
If a homomorphism $g:R\to B/J$ is given and if $b_i\in B$ is an inverse image of $g(x_i)$, the homomorphism $R[(T_i)_i]/((T_i^p-x_i)_i)\to B$ defined by $\bar\phi\circ g$ on $R$ and by $T_i\mapsto x_i$ defines a lift $\tilde g:R\to B$ of $g$ via \eqref{Eq:p-basis}.}
\end{Remark}

\begin{Remark}
\chang{In addition to the equivalences of Lemma \ref{Le:F-finite-reg},}
a noetherian $\FF_p$-algebra $R$ is regular 
iff $\phi:R\to R$ is flat by \cite[Thm.~2.1]{Kunz:Regular}. 
\end{Remark}

\begin{Defn}
\label{Def:F-nilp}
An $\FF_p$-algebra $R$ will be called $F$-nilpotent
if the kernel of the Frobenius endomorphism $\phi_R:R\to R$
is a nilpotent ideal. 
An $\FF_p$-scheme $X$ \chang{will be} called $F$-nilpotent if 
for \chang{every} affine open subscheme $U=\Spec R$ of $X$ the ring $R$ is $F$-nilpotent.
\end{Defn}

\begin{Remark}
If the kernel of $\phi_R$ is finitely generated, then $R$ is $F$-nilpotent,
in particular every noetherian $\FF_p$-algebra is $F$-nilpotent.
\end{Remark}

Let $R$ be an $\FF_p$-algebra.
For a given family $(a_i)_{i\in I}$ of elements of $R$ we consider
the rings
\begin{equation}
\label{Eq:Rapn}
R_n=R[(a_i^{1/p^n})_{i\in I}]=R[(T_i)_{i\in I}]/((T_i^{p^n}-a_i)_{i\in I})
\end{equation}
and
\begin{equation}
\label{Eq:Rapinf}
R_\infty=R[(a_i^{1/p^\infty})_{i\in I}]=\varinjlim_nR_n
\end{equation}
with respect to the inclusions $R_n\to R_{n+1}$ defined by $T_i\mapsto T_i^p$.

\begin{Lemma}
\label{Le:Fnil-remains}
If $R$ is an $F$-nilpotent\/ $\FF_p$-algebra and $a_1,\ldots,a_r\in R$,
then $R_\infty=R[(a_i^{1/p^\infty})_{1\le i\le r}]$ is $F$-nilpotent.
\end{Lemma}

\begin{proof}
Let $R_n=R[(a_i^{1/p^n})_{1\le i\le r}]$ as a subring of $R_\infty$.
The Frobenius $\phi$ of $R_n$ 
induces a homomorphism $\bar\phi:R_n\to R_{n-1}$,
and the commutative diagram
\[
\xymatrix@M+0.2em{
R_n \ar[r] \ar[d]_{\bar\phi} &
R_{n+1} \ar[d]^{\bar\phi} \\
R_{n-1} \ar[r] & R_n
}
\]
is cocartesian, 
thus $\phi:R_\infty\to R_\infty$
is the base change of $\bar\phi:R_1\to R$.
It suffices to show that the kernel of $\bar\phi$ is nilpotent.
But this $\bar\phi$ factors as
\[
R_1=R[a_i^{1/p}]\xrightarrow \psi R[(a_i^p)^{1/p}]\xrightarrow \pi R
\]
where $\psi$ is the base change of $\phi_R:R\to R$ and $\pi$
maps $(a_i^p)^{1/p}$ to $a_i$. 
The kernel of $\psi$ is nilpotent since this holds for the
kernel of $\phi_R$,
and the kernel of $\pi$
is generated by $(a_i^p)^{1/p}-a_i$ for $i=1,\ldots,r$,
thus nilpotent as well.
\end{proof}

\section{Dieudonn\'e crystals}

\label{Se:DC}

In this section we recall Dieudonn\'e crystals,
admissible filtrations, and the crystalline Dieudonn\'e functor.
As a base we take the PD scheme 
$(\Sigma,p\OOO_{\Sigma},\gamma)$ with $\Sigma=\Spec\ZZ_p$
where $\gamma$ are the canonical divided powers.
Sometimes we will also consider $\Sigma_n=\Spec\ZZ/p^n$,
viewed as a PD subscheme of $\Sigma$.

Let $X$ be a scheme on which $p$ is locally nilpotent and 
$X_0=X\times\Spec\FF_p$.
Let $\CRIS(X/\Sigma)$ be the big fppf crystalline site as in \cite{BBM}. 
Its objects are PD thickenings $(U,T,\delta)$ compatible with $\Sigma$
where $U$ is an $X$-scheme and $p$ is locally nilpotent on $T$,
and coverings of $(U,T,\delta)$ correspond to fppf coverings of $T$. 
The crystalline structure sheaf $\OOO_{X/\Sigma}$ 
on $\CRIS(X/\Sigma)$ is defined by 
$\OOO_{X/\Sigma}(U,T,\delta)=\OOO_T(T)$.
An $\OOO_{X/\Sigma}$-module $\MMM$
gives an $\OOO_T$-module $\MMM_{(U,T,\delta)}$,
in particular an $\OOO_U$-module $\MMM_{U}=\MMM_{(U,U,0)}$,
and the restriction of $\MMM$ to $\CRIS(X_0/\Sigma)$
will be denoted by $\MMM_0$.

A Dieudonn\'e crystal over $X$ is a triple $(\MMM,F,V)$ 
where $\MMM$ is a finite locally free $\OOO_{X/\Sigma}$-module
and $F:\phi^*(\MMM_0)\to\MMM_0$ and $V:\MMM_0\to\phi^*(\MMM_0)$ 
are linear maps with $FV=p$ and $VF=p$;
here $\phi$ is the Frobenius map of $X_0$.
Let $\Dcat(X)$ denote the category of Dieu\-donn\'e crystals over $X$.

For $(U,T,\delta)\in\CRIS(X_0/\Spec\FF_p)$, 
which implies that $T$ is an $\FF_p$-scheme,
the Frobenius morphism $\phi_T:T\to T$ 
has image in $U$ and thus induces a morphism $\phi_{U/T}:T\to U$;
this holds since a PD ideal in characteristic $p$ 
is annihilated by the Frobenius.
An admissible filtration for a Dieudonn\'e crystal
$(\MMM,F,V)$ over $X$ is a locally direct summand 
$\Filone\MMM_X\subseteq\MMM_X$ such that for \chang{every} 
$(U,T,\delta)\in\CRIS(X_0/\FF_p)$ we have 
\begin{equation}
\label{Eq:adm-fil}
\phi_{U/T}^*((\Filone\MMM_X)_U)=
\Ker(F:\phi^*(\MMM)_{(U,T,\delta)}\to\MMM_{(U,T,\delta)})
\end{equation}
inside $\phi_{U/T}^*(\MMM_U)=\phi^*(\MMM)_{(U,T,\delta)}$.
A variant of this definition appears in \cite[Chap.\ V]{Grothendieck:Montreal}.
The category of Dieudonn\'e crystals 
over $X$ with an admissible filtration,
also called filtered Dieudonn\'e crystals in the following,
will be denoted by
\[
\DFcat(X).
\]
The categories $\Dcat(X)$ and $\DFcat(X)$ do not
change if the fppf topology on $\CRIS(X/\Sigma)$ is replaced by
any topology finer than Zariski and coarser than fpqc, 
and one can also replace $\CRIS(X/\Sigma)$ by the corresponding small site.

\subsection*{The crystalline Dieudonn\'e functor}

Let 
\begin{equation}
\label{Eq:DDX}
\Dfun_X:\BT(X)\to\Dcat(X)
\end{equation}
be the contravariant crystalline Dieudonn\'e functor 
from $p$-divisible groups to Dieudonn\'e crystals 
as defined in \cite{BBM} and \cite{Mazur-Messing}.
For $G\in\BT(X)$ there is a natural exact sequence
of locally free $\OOO_X$-modules
\[
0\to\Lie(G)^\vee\to\Dfun_X(G)_X\to\Lie(G^\vee)\to 0,
\]
called the Hodge filtration of $G$,
which gives an extension of $\Dfun_X$ to a functor
\begin{equation}
\label{Eq:DDFX}
\DFfun_X:\BT(X)\to\DFcat(X)
\end{equation}
from $p$-divisible groups to filtered Dieudonn\'e crystals
defined by $\Filone\Dfun_X(G)_X=\Lie(G)^\vee$.
The Hodge filtration is admissible by \cite[Prop.~4.3.10]{BBM}.

\begin{Remark}[Reduction modulo $p$]
\label{Rk:red-BT-DDF}
The restriction functor of Dieuonn\'e crystals  
$\Dcat(X)\to\Dcat(X_0)$ is an equivalence, 
and lifts of a filtered Dieudonn\'e crystal
$\MMM\in\DFcat(X_0)$ 
to $\DFcat(X)$
correspond to lifts of $\Filone\MMM_{X_0}\subseteq\MMM_{X_0}$ 
to a locally direct summand of $\MMM_{(X_0,X,\gamma)}$.
If $p\ge 3$, then lifts \chang{along} the functor $\BT(X)\to\BT(X_0)$
correspond to lifts of the Hodge filtration 
by the Grothendieck-Messing Theorem \cite{Messing:Crystals},
and thus the diagram
\[
\xymatrix@M+0.2em@C+1em{
\BT(X) \ar[r]^{\DFfun_X} \ar[d] &
\DFcat(X) \ar[d] \\
\BT(X_0) \ar[r]^{\DFfun_{X_0}} &
\DFcat(X_0)
}
\]
is 2-cartesian.
In particular, if $p\ge 3$ and 
the functor $\DFfun_{X_0}$ is an equivalence,
then $\DFfun_{X}$ is an equivalence as well.
\end{Remark}

\subsection*{Forgetting the filtration}

\begin{Lemma}
\label{Le:DF-D-ff}
If $X$ is an $\FF_p$-scheme 
which is reduced or a locally noetherian l.c.i.\ scheme, 
the forgetful functor $\DFcat(X)\to\Dcat(X)$ is fully faithful.
\end{Lemma}

\begin{proof}
We have to show that a Dieudonn\'e crystal $(\MMM,F,V)$ carries at most
one admissible filtration.\footnote{\chang{Indeed, then the forgetful functor is fully faithful on isomorphisms, hence fully faithful since a homomorphism $u:M\to N$ can be encoded by the automorphism $\left(\begin{smallmatrix}1&0\\u&1\end{smallmatrix}\right)$ of $M\oplus N$.}}
If $X$ is reduced this holds since the
Frobenius endomorphism $\phi_R$ of a reduced $\FF_p$-algebra $R$ is injective,
so a direct summand of a projective $R$-module is determined by its
scalar extension \chang{along} $\phi_R$.
In the l.c.i.\ case 
we can assume that $X=\Spec R$ where $R$ is a complete local ring. 
Then $R=A/J$ where $A$ is a power series ring in finitely many variables over a field $k$ and the ideal $J$ is generated by a regular sequence $t_1,\ldots,t_r$. 
Let $J'=(t_1^p,\ldots,t_r^p)$ and $R'=A/J'$.
There are divided powers $\delta$ on the ideal $J/J'$ with 
$\delta_p(t_i)=0$.
Indeed, this is clear when $A$ equals 
$\Lambda=\FF_p[[T_1,\ldots,T_r]]$ with $t_i=T_i$,
and the general case follows since the homomorphism 
$\Lambda\to A$ defined by $T_i\mapsto t_i$ is flat.
The homomorphism $\phi_{R/R'}:R\to R'$ induced by $\phi_{R'}$
is injective because it is a base change of $\phi_A:A\to A$, 
which is faithfully flat. 
Hence a direct summand of a projective $R$-module is determined by its
scalar extension \chang{along} $\phi_{R/R'}$, and the lemma follows.
\end{proof}

\begin{Lemma}
\label{Le:DF-D-eq}
\chang{If $X$ is an $\FF_p$-scheme which locally has a $p$-basis,} 
the forgetful functor $\DFcat(X)\to\Dcat(X)$ is an equivalence.
\end{Lemma}

\begin{proof}
See \cite[Prop.~2.5.2]{Jong:Crystalline}.
We may assume that $X=\Spec R$ where $R$ has a $p$-basis $(x_i)_{i\in I}$,
and we have to show that for a Dieudonn\'e crystal $(\MMM,F,V)$
over $\Spec R$ an admissible filtration $Q\subseteq\MMM_R$ exists,
which is then unique since $X$ is reduced.
Let 
\[
Q'=\Ker(F_R:\phi_R^*(\MMM_R)\to\MMM_R).
\]
Since $R$ has a $p$-basis, every affine $(\Spec A,\Spec B,\delta)$ 
in $\CRIS(X/\Spec\FF_p)$ maps non-uniquely
to the trivial PD thickening $(X,X,0)$, 
\chang{using Remark \ref{Re:p-basis-formally-smooth}.}
Hence a direct summand $Q\subseteq\MMM_R$ is admissible iff
$\phi_R^*(Q)=Q'$.
Let $R'=R$ as an $R$-algebra via $\phi_R$,
let $R''=R'\otimes_RR'$, and let $\MMM'$ and
$\MMM''$ denote the base change of $\MMM$  
to Dieudonn\'e crystals over $\Spec R'$ and $\Spec R''$.
Then $Q'$ is an admissible
filtration for $\MMM'$, and we have to show that $Q'$ descends
to $R$, or equivalently that the scalar extensions of $Q'$
\chang{along} the two homomorphisms $R'\rightrightarrows R''$ are equal.
This holds because a Dieudonn\'e crystal over $\Spec R''$ carries
at most one admissible filtration, i.e.\
the functor $\DFcat(\Spec R'')\to \Dcat(\Spec R'')$
is fully faithful, which is a variant of Lemma \ref{Le:DF-D-ff}.
More precisely, $R''$ is isomorphic to 
$R'[\{Y_i\}_{i\in I}]/(\{Y_i^p\}_{i\in I})$
where $Y_i$ is the image of $x_i\otimes1-1\otimes x_i$; 
cf.\ Lemma \ref{Le:R''-RI} below.
Let $B=R'[\{Y_i\}_{i\in I}]/(\{Y_i^{p^2}\}_{i\in I})$.
Then the kernel of $B\to R''$ carries divided powers $\delta$
with $\delta(Y_i^p)=0$, 
and the homomorphism $\phi_{B/R''}:R''\to B$ is injective.
Hence $\DFcat(\Spec R'')\to\Dcat(\Spec R'')$ is fully faithful as required.
\end{proof}

\begin{Remark}
In general, the functor $\Dfun_X$ is not fully faithful, 
for example when $X=\Spec R$ with $R=k[X,Y]/I^2$ 
where $k$ is a field and $I=(X,Y)$; see \cite{Berthelot-Ogus}.
In this example the functor $\DFfun_X$ is not fully faithful either
because $\DFcat(X)\to\Dcat(X)$ is fully faithful.
Proof: Let $R'=k[X,Y]/I^{2p}$.
Then $R'\to R$ is a PD thickening with trivial divided powers,
and the homomorphism $\phi_{R'/R}:R\to R'$ is injective.
\end{Remark}

\section{Torsion in PD envelopes}
\label{Se:torsion-PD}

Let $R$ be a ring in which $p$ is nilpotent.
In this section we collect some permanence properties
of torsion in PD envelopes.

In the following, a presentation of $R$ will be
a surjective ring homomorphism $A\to R$ where $A$ is 
\chang{$p$-complete} 
and $\ZZ_p$-flat and $A_0=A/p$ has a $p$-basis.
For a presentation $A\to R=A/I$ we consider the \chang{$p$-completion} 
of the PD envelope relative to the PD ring $(\ZZ_p,p\ZZ_p,\gamma)$
\begin{equation}
\label{Eq:DAR}
D=D_\gamma(A\to R)^\wedge=D_{A,\gamma}(I)^\wedge.
\end{equation}
One particular choice is $A=\ZZ_p[R]^\wedge$, 
the \chang{$p$-complete} polynomial ring with set of variables $R$.
In that case we write $D=D(R)$ as in \cite{Jong-Messing}
and call $D(R)$ the universal \chang{$p$-complete} PD envelope of $R$.

It will be sufficient to consider the case where $R$ is an $\FF_p$-algebra 
because in general, with $R_0=R/p$ we have $D_\gamma(A\to R)=D_\gamma(A\to R_0)$,
and every presentation $A\to R_0$ can be lifted to a presentation $A\to R$.

\begin{Lemma}
\label{Le:D-Acris}
If $R$ is a semiperfect $\FF_p$-algebra,
i.e.\ the Frobenius \chang{homomorphism} $\phi_R:R\to R$ is surjective,
and if $A_0$ is perfect, then $D$ coincides with $A_{\cris}(R)$, 
the universal \chang{$p$-complete} PD thickening of $R$.
\end{Lemma}

\begin{proof}
One verifies that $D$ has the universal property of $A_{\cris}(R)$.
\end{proof}

\begin{Prop}
\label{Pr:PD-env-tf}
If $R$ is fixed and the ring $D$ of \eqref{Eq:DAR}
is torsion free for one presentation $A\to R$, 
then $D$ is torsion free for every presentation $A\to R$.
\end{Prop}

\begin{proof}
We can assume that $pR=0$ and
proceed in several steps.
\chang{Let us call a presentation $A\to R$ and the associated ring $D$ of polynomial type if $A$ is the $p$-completion of a polynomial algebra over $\ZZ_p$.}

{\it Step 1.}
Let $\pi:A\to R$ be a presentation
and $A'=A[(T_i)_{i\in I}]^\wedge$ for a set $I$, 
with an extension $\pi':A'\to R$ of $\pi$.
Let $D$ and $D'$ be the rings \eqref{Eq:DAR}
associated to $\pi$ and $\pi'$.
We claim that $D$ is torsion free iff $D'$ is torsion free.
Let $\tilde\pi:A'\to A$ be a homomorphism of $A$-algebras that lifts $\pi'$. 
By a change of variables we may assume that $\tilde\pi(T_i)=0$.
Then $E=D_\gamma(A'\to A)^\wedge$ is 
the \chang{$p$-complete} PD polynomial algebra
$A\langle(T_i)_{i\in I}\rangle^\wedge$.
\chang{By \cite[I Cor.~1.7.2]{Berthelot:CohCristalline}, the kernel of $ D\hatotimes_AE\to R$ carries divided powers extending those on the kernels of  $D\to R$ and $E\to A$. One deduces that $D'\cong D\hatotimes_AE=D\langle(T_i)_{i\in I}\rangle^\wedge$,
and the claim follows.}

\chang{{\it Step 2.}
As a consequence, if one choice of $D$ of polynomial type is torsion free, then every choice of $D$ of polynomial type is torsion free.}

{\it Step 3.}
Let again $A\to R$ and $D$ be given. 
Let $A_0^\infty=A_0^{\per}$ be the perfect hull of $A_0$, 
let $A^\infty=W(A_0^\infty)$ be the unique lift of $A_0^\infty$, 
and choose a homomorphism $A\to A^\infty$ that lifts $A_0\to A_0^\infty$. 
Then $A/p^r\to A^\infty/p^r$ is faithfully flat.
Let $R^\infty=R\otimes_AA^\infty$.
If $(x_i)$ is a $p$-basis of $A_0$ 
with image $(\bar x_i)$ in $R$ then $R^\infty=R[(\bar x_i^{1/p^\infty})]$
as in \eqref{Eq:Rapinf}, in particular $R^\infty$ depends only on 
the reductions $\bar x_i$.
The \chang{$p$-complete} PD envelopes $D=D_\gamma(A\to R)^\wedge$
and $D^\infty=D_\gamma(A^\infty\to R^\infty)^\wedge$ 
are related by $D^\infty=D\hatotimes_AA^\infty$,
moreover $D^\infty\cong A_{\cris}(R^\infty)$ by Lemma \ref{Le:D-Acris}.
Hence in this situation $D$ is torsion free 
iff $A_{\cris}(R^\infty)$ is torsion free, \chang{by step 1.}

{\it Step 4.}
One concludes as follows. 
\chang{Since the property that $D$ is torsion free does not change if 
polynomial variables are adjoined to $A$ by step 1, 
we can assume that a $p$-basis $(x_i)_{i\in I}$ of $A_0$ generates $R$ as a ring.
Then there is a presentation $A'\to R$ of polynomial type with a $p$-basis $(y_i)_{i\in I}$ of $A'_0$ such that $x_i$ and $y_i$ have equal image in $R$ for all $i$.
Then the rings $R^\infty$ of step~3 associated to $A\to R$ and to $A'\to R$ coincide, 
hence $D$ is torsion free iff $A_{\cris}(R^\infty)$ is torsion free
iff $D'=D_\gamma(A'\to R)^\wedge$ is torsion free by step 3.
This finishes the proof by step 2.}
\end{proof}

It will be convenient to use the following terminology.

\begin{Defn}
\label{Def:PD-tor}
A ring $R$ in which $p$ is nilpotent is PD torsion free 
if for some (equivalently: any) presentation $A\to R$ the ring $D$
of \eqref{Eq:DAR} is torsion free.
A scheme $X$ on which $p$ is locally nilpotent is PD torsion free 
if for \chang{every} affine open subscheme $\Spec R\subseteq X$ 
the ring $R$ is PD torsion free.
\end{Defn}

\begin{Remark}
\chang{$X$ is PD torsion free iff $X_0=X\times\Spec\FF_p$
has this property.}
\end{Remark}

\begin{Remark}
\label{Re:semiperf-pdtorsion}
\chang{A semiperfect $\FF_p$-algebra $R$ is PD torsion free iff the ring $A_{\cris}(R)$ is torsion free, by Lemma \ref{Le:D-Acris}.}
\end{Remark}

\begin{Prop}[\cite{Jong-Messing}]
\label{Pr:PD-tor-lci}
\chang{Every} excellent l.c.i.\ scheme $X$ over $\FF_p$ is PD torsion free.
\end{Prop}

\begin{proof}
One can assume that $X=\Spec R$ where $R$ is local.
Then the ring $D(R)$ is torsion free by 
\cite[Lemma 4.7]{Jong-Messing}.
\end{proof}

The following is implicit in the proof of \cite[Lemma 4.7]{Jong-Messing}.

\begin{Lemma}
\label{Le:DAR-lci}
Let $A\to R$ and $A'\to R'$ be presentations 
of rings in which $p$ is nilpotent 
and let $D$ and $D'$ be the associated rings \eqref{Eq:DAR}.
Assume that $A'=A[T_1,\ldots,T_n]^\wedge$ such that the
inclusion $A\to A'$ induces a homomorphism $u:R\to R'$.
If $u$ is a syntomic homomorphism, then $D/p^r\to D'/p^r$ is flat.
\end{Lemma}

\begin{proof}
We can assume that $pR=0$.

Assume first that $u$ is a relative complete intersection
with respect to the $T_i$ in the sense that
\[
R'=R[T_1,\ldots,T_n]/(f_1,\ldots,f_m)
\] 
where $f_1,\ldots,f_m$ is a regular sequence in every fibre.
We consider the rings $\tilde A=D[T_1,\ldots, T_n]$ 
and $\tilde R=\tilde A/(g_1,\ldots,g_r)$ where $g_i$ is a lift of $f_i$,
and the \chang{$p$-complete} PD envelope relative to the given PD thickening
$D\to R=D/\bar I$
\[
\tilde D=D_{(D,\bar I,\delta)}(\tilde A\to\tilde R)^\wedge.
\]
There is a natural homomorphism $\tilde R\to R'$.
We claim that the kernel of the composition $\tilde D\to \tilde R\to R'$
carries divided powers which identify $\tilde D$ with $D'$. 

Let $A_r=A/p^r$ etc.
Then $\Spec\tilde R_r\to\Spec\tilde A_r$ 
is a regular embedding of flat $D_r$-schemes of finite \chang{presentation},
which implies that $\tilde D_r$ is flat over $D_r$ by \cite[Lemma 2.3.3]{BBM} and its proof.
Let $J_r=\Ker(\tilde D_r\to\tilde R_r)$ and $I_r=\Ker(D_r\to R)$.
The divided powers on $I_r$ extend to 
$I_r\tilde D_r=I_r\otimes_{D_r}\tilde D_r$,
and we have $J_r\cap I_r\tilde D_r=I_r\otimes_{D_r}J_r$.
It follows that the divided powers on $I_r\tilde D_r$
and on $J_r$ are compatible and extend to $I_r\tilde D_r+J_r$,
which is the kernel of $\tilde D_r\to R'$.
One verifies that the resulting PD thickening $\tilde D_r\to R'$
has the universal property of $D'_r\to R'$, which proves the claim.
Since $\tilde D_r$ is flat over $D_r$, the lemma is proved when
$u$ is a relative complete intersection.

In general one can cover $\Spec R'$ 
by open sets $\Spec R'_i$ with $R'_i=R'_{h_i}$,
equipped with the presentation $A'_i=A'[T_{n+1}]^\wedge\to R'_i$
defined by $T_{n+1}\mapsto h_i$,
such that $R\to R'_i$ is a relative complete intersection with respect to
$T_1,\ldots,T_{n+1}$. 
Let $D'_i$ be the ring \eqref{Eq:DAR} associated to $A'_i\to R'_i$.
Then by the first part of the proof, 
$D\to D'_i$ is flat modulo $p^r$ 
and $D'\to\prod D'_i$ is faithfully flat modulo $p^r$. 
Hence $D\to D'$ is flat mod $p^r$.
\end{proof}

\begin{Cor}
\label{Co:DAR-lci-tf}
Let $R\to R'$ be a faithfully flat syntomic homomorphism 
of rings in which $p$ is nilpotent.
Then $D(R)/p^r\to D(R')/p^r$ is faithfully flat. 
In particular, $R$ is PD torsion free iff this holds for $R'$. 
\end{Cor}

\begin{proof}
Note that $R\subseteq R'$ is a subring.
Let $A=\ZZ_p[R]^\wedge$ and $D=D(R)$. 
For a finite subset $I\subseteq R'\setminus R$ 
which generates $R'$ as an $R$-algebra let $A_I=\ZZ_p[R\cup I]^\wedge$ 
and let $D_I$ be the ring \eqref{Eq:DAR} for the natural presentation $A_I\to R'$.
Then $D/p^r\to D_I/p^r$ is flat by Lemma \ref{Le:DAR-lci},
and $D(R')/p^r$ is the filtered colimit over $I$ of $D_I/p^r$,
thus flat over $D/p^r$.
\end{proof}

\begin{Remark}
\label{Re:pbasis-pd-tf}
\chang{Every $\FF_p$-scheme $X$ which locally has a $p$-basis is PD torsion free. Indeed, one can assume that $X=\Spec R$ where $R$ has a $p$-basis, which allows a presentation $A\to R$ with $A/p=R$ by \cite[Prop.~1.1.7]{Berthelot-Messing} or \cite[Lemmas 1.1.2 \& 1.2.2]{Jong:Crystalline}, and then $D=A$.}
\end{Remark}

\section{Frames and windows}
\label{Se:frames}

We use the terminology of frames and windows
as in \cite{Zink:Windows} and \cite{Lau:Frames}
with some modifications adapted to the present situation.

\begin{Defn}
\label{Def:frame}
A frame $\u A=(A,I,R,\sigma,\sigma_1)$
consists of \chang{$p$-complete} rings $A$ and $R=A/I$ for an ideal $I\subseteq A$,
a ring homomorphism $\sigma:A\to A$ which induces the Frobenius
on $A/p$, and a $\sigma$-linear map $\sigma_1:I\to A$ with $p\sigma_1=\sigma$.

An $\u A$-window $\u M=(M,M_1,\Phi,\Phi_1)$ 
consists of a finite projective $A$-module $M$, 
a submodule $M_1\subseteq M$ with $IM\subseteq M_1$ such that
$M/M_1$ is projective over $R$,
and $\sigma$-linear maps $\Phi:M\to M$ and $\Phi_1:M_1\to M$ 
with $p\Phi_1=\Phi$ and $\Phi_1(ax)=\sigma_1(a)\Phi(x)$ 
for $a\in A$ and $x\in M$,
such that $M$ is generated by $\Phi(M)+\Phi_1(M_1)$.

Homomorphisms of $\u A$-windows are module homomorphisms
that preserve all data. 
The category of $\u A$-windows is denoted by $\Win(\u A)$.
\end{Defn}

\begin{Remark}
\label{Rk:Psi}
For an $\u A$-window $\u M$ as in Definition \ref{Def:frame}
there is a decomposition $M=L\oplus T$ with $M_1=L\oplus IT$, 
called a normal decomposition,
and the pair $(\Phi,\Phi_1)$ corresponds
to a $\sigma$-linear isomorphism $\Psi:L\oplus T\to M$
defined by $\Psi=\Phi_1$ on $L$ and $\Psi=\Phi$ on $T$.
Proof:
The existence of $\sigma_1$ implies that $x^p=0$ 
for \chang{all} $x\in\Ker(A/p\to R/p)$,
hence finite projective $R$-modules can be lifted to $A$,
and the decomposition of $M$ follows.
See \cite[Lemma 2.6]{Lau:Frames} for more details.
It also follows that $I+pA\subseteq\Rad(A)$,
which means that $\u A$ is a frame in the sense of \cite[Def.~2.1]{Lau:Frames}.
\end{Remark}

\begin{Remark}[Functoriality]
\label{Rk:frame-funct}
Let $\u A'=(A',I',R',\sigma,\sigma_1)$ be another frame.
A frame homomorphism $\alpha:\u A\to\u A'$ is a ring homomorphism
$A\to A'$ that induces a homomorphism $R\to R'$ and commutes with $\sigma$ and $\sigma_1$.
It induces a base change functor $\alpha^*:\Win(\u A)\to\Win(\u A')$;
see \cite[Lemma 2.10]{Lau:Frames}.
Universal property:
If windows $\u M$ over $\u A$ and $\u M'$ over $\u A'$ are given,
define an $\alpha$-homomorphism $\u M\to\u M'$ to be a homomorphism of $A$-modules
that preserves the filtration and commutes with $\Phi$ and $\Phi_1$.
Then $\alpha$-homomorphisms $\u M\to\u M'$ correspond to
homomorphisms of $\u A'$-windows $\alpha^*\u M\to\u M'$. 
\end{Remark}

\begin{Example}
\label{Ex:frame-W}
For a \chang{$p$-complete} ring $R$ the ring of $p$-typical Witt vectors
$W(R)$ carries a frame structure $\u W(R)=(W(R),I(R),R,\sigma,\sigma_1)$
where $\sigma$ is the Witt vector Frobenius and $\sigma_1$ is the inverse
of the bijective Verschiebung homomorphism $v:W(R)\to I(R)$.
A window over $\u W(R)$ is a 3n-display over $R$ 
in the sense of \cite{Zink:Display}, 
\chang{commonly called display in later literature.}
\end{Example}

\begin{Example}
\label{Ex:Witt-frame}
Let $R$ be a \chang{$p$-complete} ring.
A frame for $R$ in the sense of \cite{Zink:Windows} is a 
\chang{$p$-complete} PD thickening $A\to R=A/I$ where $A$ is torsion free,
with a Frobenius lift $\sigma$ on $A$.
Then $\sigma(I)\subseteq pA$, and $(A,I,R,\sigma,\sigma_1)$
with $\sigma_1=p^{-1}\sigma$ is a frame in the sense of
Definition \ref{Def:frame}.
\end{Example}

Let us recall a version of the deformation lemma for windows.
An endomorphism $f$ of an abelian group is called \chang{locally}
nilpotent if \chang{any} element is annihilated by some power of $f$.

\begin{Prop}
\label{Pr:deform-win}
Let $\alpha:\u A'\to\u A$ be a homomorphism of frames 
such that $A'\to A$ is surjective with kernel $N$
and $R'\to R$ is bijective.
If $\sigma_1:N\to N$ induces a \chang{locally} nilpotent endomorphism of $N/p$, 
then the base change functor $\alpha^*:\Win(\u A')\to\Win(\u A)$
is an equivalence.
\end{Prop}

\begin{proof}
Since $A$ and $A'$ are \chang{$p$-complete}, 
$N$ is closed and hence complete for the $p$-adic topology of $A'$.
It follows that $N$ is \chang{$p$-complete}; 
see the proof of 
\cite[\href{http://stacks.math.columbia.edu/tag/090T}{Tag 090T}]{Stacks-Project}.
For $n\ge 0$ we define a frame $\u B{}_n=(A'/p^nN,I'/p^nN,R',\sigma,\sigma_1)$
as a quotient of $\u A'$.
Since $\sigma$ is \chang{locally} nilpotent on $N/p^{n}N$, the
projection homomorphism $\u B{}_{n}\to\u B{}_0=\u A$ 
induces an equivalence of windows by \cite[Thm.~3.2]{Lau:Frames}.
Since $\u A'=\varprojlim_n\u B{}_n$ taken componentwise,
the result follows by \cite[Le.~2.12]{Lau:Frames}.
\end{proof}

We will also use a sheaf version of frames and windows.

\begin{Defn}
\label{Def:frame-topos}
Let $T$ be a topos. 

A ring $\BBB$ in $T$ is called \chang{$p$-complete} if 
$\BBB\xrightarrow\sim\varprojlim_r\BBB/p^r$.

A frame in $T$ is a collection
$\u\AAA=(\AAA,\III,\RRR, \sigma,\sigma_1)$
where $\AAA$ and $\RRR$ are \chang{$p$-complete} rings in $T$
with $\RRR=\AAA/\III$ for an ideal $\III\subseteq \AAA$,
$\sigma:\AAA\to \AAA$ is a ring homomorphism 
which induces the Frobenius homomorphism on $\AAA/p$,
and $\sigma_1:\III\to \AAA$ is a $\sigma$-linear map 
with $p\sigma_1=\sigma$.

If $\u\AAA$ is a frame in $T$, 
an $\u\AAA$-window is a collection $\u\MMM=(\MMM,\MMM_1,\Phi,\Phi_1)$
where $\MMM$ is a finite locally free $\AAA$-module,
$\MMM_1\subseteq \MMM$ is a submodule with $\III\MMM\subseteq \MMM_1$
such that $\MMM/\MMM_1$ is locally free over $\RRR$,
and $\Phi:\MMM\to \MMM$ and $\Phi_1:\MMM_1\to \MMM$ are $\sigma$-linear maps
with $p\Phi_1=\Phi$ on $\MMM_1$ and $\Phi_1(ax)=\sigma_1(a)\Phi(x)$ 
for local sections $a\in \III$ and $x\in \MMM$,
such that $\MMM$ is generated by $\Phi(\MMM)+\Phi_1(\MMM_1)$.

Homomorphisms of $\u \AAA$-windows are module homomorphisms
that preserve all data. 
The category of $\u \AAA$-windows is denoted by $\Win(\u \AAA/T)$.
\end{Defn}

Remark \ref{Rk:Psi} carries over to the sheaf version as follows.

\begin{Lemma}
\label{Le:Psi}
For $\AAA$-modules $\MMM_1\subseteq \MMM$ 
as in Definition \ref{Def:frame-topos},
locally there is a decomposition of $\AAA$-modules 
$\MMM=\LLL\oplus \TTT$ with $\MMM_1=\LLL\oplus \III\TTT$. 
If such a decomposition is given, 
pairs $(\Phi,\Phi_1)$ such that $(\MMM,\MMM_1,\Phi,\Phi_1)$ is an
$\u \AAA$-window correspond to $\sigma$-linear isomorphisms 
$\Psi:\LLL\oplus \TTT\to \MMM$ defined by 
$\Psi=\Phi_1$ on $\LLL$ and $\Psi=\Phi$ on $\TTT$.
\end{Lemma}

\begin{proof}
The existence of $\sigma_1$ implies that
\chang{every} $x$ in the kernel of $\AAA/p\to \RRR/p$
satisfies $x^p=0$.
It follows that the map of idempotents
$\Idem(M_n(\AAA))\to\Idem(M_n(\RRR))$ is surjective as sheaves.
Indeed, let $e\in M_n(\RRR)$ be idempotent.
Locally $e$ can be lifted to an element $y\in M_n(\AAA)$,
with image $y_r\in M_n(\AAA/p^r)$.
Since the kernel of $\AAA/p^r\to \RRR/p^r$ is a nil-ideal,
there is a unique idempotent $\tilde e_r\in M_n(\AAA/p^r)$ 
which is a polynomial in $y_r$ with integral coefficients
such that $\tilde e_r$ and $e$ have equal image in $M_n(\RRR/p^r)$.
The system $(\tilde e_r)_r$ gives an idempotent 
$\tilde e\in M_n(\AAA)$ which maps to $e$.
Now the surjective map of finite locally free $\RRR$-modules 
$\MMM/\III\MMM\to \MMM/\MMM_1$
splits locally, and the resulting decomposition of $\MMM/\III\MMM$
lifts to a decomposition of $\MMM$ locally.
For the correspondence between $(\Phi,\Phi_1)$ and $\Psi$ see
\cite[Lemma 2.6]{Lau:Frames}.
\end{proof}

Definition \ref{Def:frame} and Definition \ref{Def:frame-topos} 
are related as follows.
Let $\LF(\AAA)$ denote the category of finite locally free $\AAA$-modules.

\begin{Lemma}
\label{Le:win-win}
Let $\u\AAA$ be a frame in $T$ as in Definition \ref{Def:frame-topos} 
and consider the \chang{collection of} global sections $\u A=\Gamma(T,\u\AAA)$
defined by $A=\Gamma(T,\AAA)$ and $R=\Gamma(T,\RRR)$ etc.
If $A\to R$ is surjective, 
then $\u A$ is a frame as in Definition \ref{Def:frame}.
If moreover the functor of global sections induces equivalences
$\LF(\AAA)\xrightarrow\sim($finite projective $A$-modules)
and $\LF(\RRR)\xrightarrow\sim($finite projective $R$-modules),
then it induces an equivalence $\Win(\u\AAA/T)\cong\Win(\u A)$.
\end{Lemma}

\begin{proof}
Let us show that the ring $A$ is \chang{$p$-complete}.
Let $A_n=\Gamma(T,\AAA/p^n)$ and let $A_n'$ be the image of $A\to A_n$.
Since $\AAA$ is \chang{$p$-complete} we have $A\cong\varprojlim A_n$,
hence $A\cong\varprojlim A_n'$.
Since $p^nA$ lies in the kernel of $A\to A_n'$
it follows that $A$ is \chang{$p$-complete}; see the proof of 
\cite[\href{http://stacks.math.columbia.edu/tag/090T}{Tag 090T}]{Stacks-Project}.
Similarly $R$ is \chang{$p$-complete}, and $\u A$ is a frame if $A\to R$ is surjective.
Assume that the hypothesis on finite locally free modules holds.
For given $\MMM\in\LF(\AAA)$ let  
$\bar\MMM=\MMM\otimes_\AAA\RRR\in\LF(\RRR)$,
and consider the associated finite projective modules 
$M=\Gamma(T,\MMM)$ over $A$ and $\bar M=\Gamma(T,\bar\MMM)$ over $R$.
Then $\MMM=M\otimes_A\AAA$ and $\bar\MMM=\bar M\otimes_R\RRR$,
moreover $\bar M=M\otimes_AR$ since this holds after $\otimes_R\RRR$.
It follows that submodules $\MMM_1\subseteq\MMM$ such that
$\MMM/\MMM_1$ is a finite locally free $\RRR$-module correspond to
submodules $M_1\subseteq M$ such that $M/M_1$ is 
finite projective over $R$, via $M_1=\Gamma(T,\MMM_1)$.
If $M_1\subseteq M$ is given, 
choose a decomposition $M=L\oplus T$ with $M_1=L\oplus IT$,
which gives $\MMM=\LLL\oplus\TTT$ with $\MMM_1=\LLL\oplus\III\TTT$
for $\LLL=L\otimes_A\AAA$ and $\TTT=T\otimes_A\AAA$.
Then window structures on $(\MMM,\MMM_1)$ and on $(M,M_1)$
correspond to $\sigma$-linear isomorphisms 
$\Psi:\LLL\oplus\TTT\to\MMM$ and $\Psi':L\oplus T\to M$
by Remark \ref{Rk:Psi} and Lemma \ref{Le:Psi},
and the relation $\Psi'=\Gamma(T,\Psi)$
gives a bijective correspondence between the window structures.
\end{proof}

Finally let us record an elementary fact.

\begin{Lemma}
\label{Le:win-phi-mod}
Let $\u\AAA$ be a frame in a topos as in Definition \ref{Def:frame-topos}
such that $p\in\III$ and $\sigma_1(p)=1$. 
For $\u\AAA$-windows $\u\MMM$, $\u\NNN$ the natural homomorphism
$\rho:\Hom(\u\MMM,\u\NNN)\to\Hom_{\AAA,\Phi}(\MMM,\NNN)$ 
from window homomorphisms to $\Phi$-module homomorphisms
is injective with cokernel annihilated by $p$.
\end{Lemma}

\begin{proof}
Clearly $\rho$ is injective.
Assume that $g:\MMM\to\NNN$ commutes with $\Phi$ and let $h=pg$.
Then $h(\MMM_1)\subseteq p\NNN\subseteq\NNN_1$, 
and $h$ commutes with $\Phi_1$ since for $x\in\MMM_1$ we have
$\Phi_1(h(x))=\Phi_1(pg(x))=\sigma_1(p)\Phi(g(x))=\Phi(g(x))$ and
$h(\Phi_1(x))=g(\Phi(x))$ using that $p\Phi_1=\Phi$.
\end{proof}

\subsection*{\chang{Divided Frobenius on PD envelopes}}

\chang{We use the notion of $\delta$-rings of \cite{Bhatt-Scholze:Prisms}. Recall that $\gamma$ denotes the divided powers on the ideal $p\ZZ_p\subseteq\ZZ_p$. Our primary source of frames is the following.\footnote{An earlier version of this article used a weaker statement with a longer proof.}

\begin{Lemma}
\label{Le:delta-pd-sigma1}
Let $A$ be a $\delta$-ring over $\ZZ_p$ and $R=A/I$ for an ideal $I$. Let $D=D_{\gamma}(A\to R)$ be the pd envelope relative to $\gamma$, and let $J\subseteq D$ be the natural PD ideal. Then $D$ carries a natural structure of $\delta$-ring, which yields a Frobenius lift $\sigma:D\to D$, and there is a natural $\sigma$-linear map $\sigma_1:J\to D$ that satisfies $p\sigma_1=\sigma|_J$.
\end{Lemma}

\begin{proof}
Cf.\ \cite[Remark 4.12]{Hoff-Lau:Sheared}. 
The $\delta$-structure on $A$ corresponds to a homomorphism $A\to W_2(A)$.
 The kernel of $W_2(D)\to R$ carries natural divided powers compatible with $\gamma$ by \cite[\S2.3]{Zink:Display}.
Hence the composition $A\to W_2(A)\to W_2(D)$ extends to a PD homomorphism $D\to W_2(D)$, which gives a $\delta$-structure on $D$. For $a\in J$ we set $\sigma_1(a)=(p-1)!a^{[p]}+\delta(a)$. Clearly $p\sigma_1=\sigma$. One verifies that $\sigma_1$ is a $\sigma$-linear map by a calculation or by passing to a universal situation where $D$ is torsion free. 
\end{proof}

\begin{Cor}[{\cite[Lemma 4.1.8]{Scholze-Weinstein}}]
\label{Co:SW}
For semiperfect $\FF_p$-algebras $R$ there is a unique functorial $\sigma$-linear map $\sigma_1:I_{\cris}(R)\to A_{\cris}(R)$ such that $p\sigma_1=\sigma$, moreover $\sigma_1(p)=1$ and $\sigma_1(\gamma_n([a]))=\frac{(pn)!}{p\cdot n!}\gamma_{pn}([a])$ for $a\in\ker(R^\flat\to R)$.
\end{Cor}

\begin{proof}
Lemma \ref{Le:delta-pd-sigma1} with $A=W(R^\flat)$ gives a functorial $\sigma_1$ with $p\sigma_1=\sigma$. Uniqueness and the remaining formulas hold after multiplication by a power of $p$, hence in general since in the universal case $A_{\cris}$ is torsion free.
\end{proof}
}

\section{Dieudonn\'e theory over semiperfect rings}
\label{Se:semiperfect}

Let $R$ be a semiperfect $\FF_p$-algebra,
i.e.\ the Frobenius $\phi_R:R\to R$ is surjective.
The ring $R^\flat=\varprojlim(R,\phi_R)$ is perfect,
and the projection $R^\flat\to R$ is surjective, thus 
\[
R=R^\flat/J
\]
for an ideal $J$.
The ring $A_{\inf}(R)=W(R^\flat)$ is the universal pro-infinitesimal
thickening of $R$, and $A_{\cris}(R)=D_\gamma(A_{\inf}(R)\to R)^\wedge$
as in \eqref{Eq:DAR} is the universal \chang{$p$-complete} PD thickening of $R$.
Let $I_{\cris}(R)$ be the kernel of $A_{\cris}(R)\to R$.
The Frobenius $\phi_R$ induces an endomorphism $\sigma$ of $A_{\cris}(R)$.
\chang{Corollary \ref{Co:SW} yields} a functorial frame
\begin{equation}
\label{Eq:uAcrisR}
\u A{}_{\cris}(R)=(A_{\cris}(R),I_{\cris}(R),R,\sigma,\sigma_1).
\end{equation}
\chang{By \cite[Th.~6.3]{Lau:Semiperfect} there is a contravariant functor
from $p$-divisible groups over $R$ to windows over this frame,
\begin{equation}
\label{Eq:PhiRcris}
\Phi_R^{\cris}:\BT(\Spec R)\to\Win(\u A{}_{\cris}(R)),
\end{equation}
which is functorial in $R$
such that for $\u M=\Phi_R^{\cris}(G)$ the triple
$(M,\Filone M,F)$ is the value of the filtered Dieudonn\'e crystal $\DFfun(G)$ at $A_{\cris}(R)$, and this determines $\Phi_R^{\cris}$.} 
We recall the following facts:

\begin{Remark}
By \cite[Th.~7.10]{Lau:Semiperfect}
the functor $\Phi_R^{\cris}$ is an equivalence if $R$ is iso-balanced.
The ring $R$ is called balanced if $\phi(J)=J^p$,
or equivalently if the ideal $\bar J=\Ker(\phi_R)$ satisfies $\bar J^p=0$, 
and $R$ is called iso-balanced 
if there is an ideal $I\subseteq R$ 
annihilated by a power of $\phi$ such that $R/I$ is balanced;
then $I$ is nilpotent.
If $J$ is finitely generated then $R$ is iso-balanced.
\end{Remark}

\begin{Remark}
\label{Rk:Acris-DF}
If the ring $A_{\cris}(R)$ is torsion free,
the category $\Win(\u A{}_{\cris}(R))$
is equivalent to the category of filtered Dieudonn\'e crystals 
$\DFcat(\Spec R)$,
and the functor $\Phi_R^{\cris}$ 
corresponds to $\DFfun_{\Spec R}$;
see  \cite[Prop.~2.6.4]{Cais-Lau}.
This holds when $R$ is a complete intersection
in the sense that $J$ is generated by a regular sequence.
In that case $R$ is iso-balanced, so it follows that
the functor $\DFfun_{\Spec R}$ is an equivalence.
\end{Remark}

We will generalize these results in two directions.

\subsection{$F$-nilpotent semiperfect rings}

Let $R=R^\flat/J$ be a semiperfect $\FF_p$-algebra as above.
The ring $R$ is $F$-nilpotent in the sense of Definition \ref{Def:F-nilp}
iff $J^{p^r}\subseteq\phi(J)$ for some $r\ge 0$. 
One verifies that every iso-balanced semiperfect ring is $F$-nilpotent.

Assume that $J=K_0\supseteq K_1\supseteq\ldots$ 
is a decreasing sequence of ideals of $R^\flat$ 
such that $K_i^p\subseteq K_{i+1}$; 
this is called an admissible sequence of ideals 
in \cite[Def.~7.5]{Lau:Semiperfect}. 
Then the set 
\[
W(K_*)=\{(a_0,a_1,\ldots)\in W(R^\flat)\mid a_i\in K_i\}
\]
is an ideal of $W(R^\flat)$,
and the ring $A(K_*)=W(R^\flat)/W(K_*)$ 
is a straight weak lift of $R$ in the sense of 
\cite[Def.~7.3]{Lau:Semiperfect} by \cite[Lemma 7.6]{Lau:Semiperfect},
which implies that it carries a natural frame structure
\[
\u A(K_*)=(A(K_*),pA(K_*),R,\sigma,\sigma_1)
\] 
as a quotient of the frame $\u W(R^\flat)$ of Example \ref{Ex:frame-W}.
By \cite[Lemma 7.4]{Lau:Semiperfect}, 
\chang{the unique homomorphism $\varkappa:A_{\cris}(R)\to A(K_*)$ of PD thickenings of $R$ is a } frame homomorphism
\[
\varkappa:\u A{}_{\cris}(R)\to\u A(K_*).
\]
The composition of $\Phi_R^{\cris}$ with the base change functor 
$\varkappa^*$ is a functor 
\[
\Phi_{K_*}:\BT(R)\to\Win(\u A(K_*)).
\]
The minimal admissible sequence $J_*$ is defined by $J_i=J^{p^i}$.

\begin{Prop}
\label{Pr:kappa-equiv}
If $R$ is $F$-nilpotent and $K_*=J_*$ is the minimal admissible sequence, 
$\varkappa$
induces an equivalence of the window categories.
\end{Prop}

\begin{proof}
This is an application of Proposition \ref{Pr:deform-win}.
\chang{Clearly} $\varkappa$ is surjective.
Let $N\subseteq A_{\cris}(R)$ be the kernel of $\varkappa$.
We have to show that $\sigma_1$ induces a \chang{locally} nilpotent
endomorphism of $N/pN$. 
Let $N_0\subseteq A_{\cris}(R)$ be the ideal 
generated by the elements
$[a]^{[n]}$ for $a\in J$ and $n\ge 1$ (not completed),
where the exponent $^{[n]}$ denotes the $n$-th divided power.
Then $N_0\subseteq N$.
\chang{Any} $x\in A_{\cris}(R)$ can be written as $x=b+\sum_{i\ge 0}p^iy_i$
with $b\in W(R^\flat)$ and $y_i\in N_0$,
and we have $x\in N$ iff $b\in N$ iff $b\in W(J_*)$.
It follows that $N/pN$ is generated by $N_0$ and $W(J_*)$.
For $a\in J$ we have
\begin{equation}
\label{Eq:sigma1-an}
\sigma_1([a]^{[n]})=c_n[a]^{[np]}
\qquad
\text{with} 
\qquad
c_n=(np)!/(n!\, p)
\end{equation}
by \chang{Corollary \ref{Co:SW}},
and $c_n$ is divisible by $p$ if $n\ge p$, 
so $\sigma_1^2$ is zero on $N_0/pN_0$.
Since $R$ is $F$-nilpotent we have 
$J^{p^r}\subseteq\phi(J)$ for some $r$,
and thus $J^{p^{r+s}}\subseteq\phi(J^{p^s})$ for \chang{all} $s$.
Hence for $a\in W(J_*)\cap p^rW(R^\flat)$ we have
$\sigma_1^r(a)=\sigma(b)=p\sigma_1(b)$ with $b\in W(J_*)$
and thus $\sigma_1^r(a)=0$ in $N/pN$.
Finally, for $a\in J_i$ with $0\le i<r$ the element $v^i([a])$ of $W(J_*)$
is mapped to $pN_0$ by $\sigma_1^{i+2}$ by \eqref{Eq:sigma1-an} again.
It follows that $\sigma_1^{r+1}$ is zero on $N/pN$.
\end{proof}

Let $R_1=R^\flat/J^p$ and $A=A(J_*)$ and $A_1=A(J_{*+1})$.
Then $R_1$ is semiperfect with
$R_1^\flat=R^\flat$, and $A_1$ gives a frame $\u A{}_1$ for $R_1$.

\begin{Prop}
\label{Pr:BT-R1-R}
There is a $2$-cartesian diagram of categories:
\[
\xymatrix@M+0.2em{
\BT(R_1) \ar[r] \ar[d]_{\Phi_{J_{*+1}}} & 
\BT(R) \ar[d]^{\Phi_{J_{*}}} \\ 
\Win(\u A{}_1) \ar[r] &
\Win(\u A)
}
\] 
\end{Prop}

\begin{proof}
There are natural surjective homomorphisms $R_1\to R$
and $\u A{}_1\to\u A$ which induce the horizontal functors.
Let $I$ be the kernel of $A_1\to A\to R$.
We will define divided powers on $I$ and
a homomorphism $\sigma_1:I\to A_1$ which gives a frame
\[
\u A'=(A_1,I,R,\sigma,\sigma_1).
\]
Let $N$ be the kernel of $A_1\to A$.
Then $I=N+pA_1$, and
\[
N=W(J_*)/W(J_{*+1}),
\qquad
N\cap pA_1=v(W(J_{*+1}))/v(W(J_{*+2})).
\]
It follows that $\sigma_1:pA_1\to A_1$ is zero on $N\cap pA_1$
and thus extends uniquely to a homomorphism
$\sigma_1:I\to A_1$ with $\sigma_1(N)=0$.
For $a_1,\ldots,a_p\in W(J_*)$ we have $a_1\cdots a_p\in W(J_{*+1})$ 
by the homogeneity properties of the Witt vector multiplication. 
Hence $N^p=0$, 
and $N$ carries the trivial divided powers $\delta$ defined by
$\delta_p(x)=0$ for all $x\in N$.
For $a\in W(J_{*+1})$, the natural divided powers $\gamma$
on $pW(R^\flat)=I(R^\flat)$ give
$\gamma_p(v(a))=(p^{p-1}/p!)v(a^p)\in W(J_{*+1})$ 
since $a^p\in W(J_{*+2})$. 
It follows that the given divided powers $\gamma$ on $pA_1$
and $\delta$ on $N$ coincide on $N\cap pA_1$ and thus extend 
to divided powers on $I$.

There are natural frame homomorphisms $\u A{}_1\to\u A'\to \u A$ 
over the ring homomorphisms $R_1\to R=R$.
Since $\sigma_1$ is zero on $N$,
the homomorphism $\u A'\to\u A$ induces an equivalence of windows
by Proposition \ref{Pr:deform-win}.

Next we want to define a frame homomorphism $\u A{}_{\cris}(R)\to\u A'$.
The projection $W(R^\flat)\to A_1$ extends to a homomorphism 
$\tilde\varkappa:A_{\cris}(R)\to A_1$ 
of \chang{$p$-complete} PD thickenings of $R$
due to the divided powers on $I$.
The composition of $\tilde\varkappa$ with either
$A_1\to A$ or $A_{\cris}(R_1)\to A_{\cris}(R)$
is the homomorphism $\varkappa$ of $A$ or $A_1$.
The homomorphism $\tilde\varkappa$ commutes with $\sigma$
because $\sigma$ preserves the divided powers on both sides.
We claim that $\tilde\varkappa$ is a frame homomorphism
$\u A{}_{\cris}(R)\to\u A'$, which means that $\tilde\varkappa$
commutes with $\sigma_1$. This is a direct calculation:
The ideal $I_{\cris}(R)$ is generated
by $p$ and the elements $[a]^{[n]}$ for $a\in J$ and $n\ge 1$,
where $^{[n]}$ means $n$-th divided power.
Using \eqref{Eq:sigma1-an} we get
\[
\tilde\varkappa(\sigma_1([a]^{[n]}))
=c_n\tilde\varkappa([a]^{[np]})
=c_n\tilde\varkappa([a])^{[np]}
=0
\]
since $\tilde\varkappa([a])\in N$ and $N^{[p]}=0$, 
moreover
\[
\sigma_1(\tilde\varkappa([a]^{[n]}))
=\sigma_1(\tilde\varkappa([a])^{[n]})=0
\]
since $\tilde\varkappa([a])^{[n]}\in N$ and $\sigma_1(N)=0$.

The composition of $\Phi_R^{\cris}$ with the
base change functor $\tilde\varkappa^*$
is a functor $\Phi'$ such that the following diagram commutes.
\[
\xymatrix@M+0.2em{
\BT(R_1) \ar[r] \ar[d]^{\Phi_{J_{*+1}}} & 
\BT(R) \ar[r]^{\id} \ar[d]^{\Phi'} &
\BT(R) \ar[d]^{\Phi_{J_{*}}} \\
\Win(\u A{}_1) \ar[r] &
\Win(\u A') \ar[r]^\sim &
\Win(\u A)
}
\]

For a $p$-divisible group $G$ over $R$ and 
$\u M'=\Phi'(G)$, the module $M'\otimes_{A_1}R_1$
coincides with the value of the Dieudonn\'e crystal $\Dfun(G)_{R_1\to R}$,
where $R_1\to R$ is a PD thickening induced by the divided powers
on $I\subseteq A_1$.
The divided powers on $\Ker(R_1\to R)$ are trivial since
$N$ maps surjectively to this ideal.
By the Grothendieck-Messing Theorem
it follows that lifts of $G$ to $R_1$ and lifts of $\u M'$ to
$\u A{}_1$ correspond to lifts of the Hodge filtration in the same way,
and the proposition follows.
\end{proof}

\begin{Cor}
\label{Co:BT-Rflat-R}
If $R$ is $F$-nilpotent, 
there is a $2$-cartesian diagram of categories
\[
\xymatrix@M+0.2em{
\BT(R^\flat) \ar[r] \ar[d]_{\Phi^{\cris}_{R^\flat}} & 
\BT(R) \ar[d]^{\Phi_{J_{*}}} \\ 
\Win(\u W(R^\flat)) \ar[r] &
\Win(\u A)
}
\] 
\end{Cor}

\begin{proof}
Proposition \ref{Pr:BT-R1-R} applied to
$R_n=R^\flat/J^{p^n}$ for $n\ge 0$ gives $2$-cartesian squares
\[
\xymatrix@M+0.2em{
\BT(R_n) \ar[r] \ar[d]_{\Phi_{J_{*+n}}} & 
\BT(R) \ar[d]^{\Phi_{J_{*}}} \\ 
\Win(\u A{}_n) \ar[r] &
\Win(\u A)
}
\] 
with $A_n=W(R^\flat)/W(J_{*+n})$.
Since $R$ is $F$-nilpotent we have $\phi(J)\subseteq J^{p^r}$ for
some $r$, which implies that $R^\flat=\varprojlim_n R^\flat/J^{p^n}$
and thus $W(R^\flat)=\varprojlim_n A_n$.
The proposition follows since
the categories of $p$-divisible groups and windows 
preserve these limits 
by \cite[Lemma~2.12]{Lau:Frames} and the obvious analogue of
\cite[II Lemma~4.16]{Messing:Crystals}; 
see also \cite[Lemma~2.4.4]{Jong:Crystalline}.
\end{proof}

\begin{Thm}
\label{Th:semiperfect-Fnil}
For \chang{every} $F$-nilpotent semiperfect $\FF_p$-algebra $R$ the functor
\[
\Phi_R^{\cris}:\BT(R)\to\Win(\u A{}_{\cris}(R))
\]
is an equivalence.
\end{Thm}

\begin{proof}
The functor $\Phi_{R^\flat}:\BT(R^\flat)\to\Win(\u W(R^\flat))$
is an equivalence by a result of Gabber; 
see \cite[Thm.\ 6.4]{Lau:Smoothness}.
The reduction functor $\Win(\u W(R^\flat))\to\Win(\u A)$
is essentially surjective
by the proof of \cite[Thm.~5.7]{Lau:Semiperfect}. 
The $2$-cartesian square of Corollary \ref{Co:BT-Rflat-R} 
implies that $\Phi_{J_*}:\BT(R)\to\Win(\u A)$ is an equivalence,
using \cite[Lemma 5.9]{Lau:Semiperfect}. 
Proposition \ref{Pr:kappa-equiv} gives the result.
\end{proof}

\begin{Cor}
\label{Co:semiperfect-Fnil}
\chang{For every $F$-nilpotent semiperfect $\FF_p$-algebra $R$ which is PD torsion free (Definition \ref{Def:PD-tor}), the functor $\DFfun_{\Spec R}$ is an equivalence.}
\end{Cor}

\begin{proof}
Use Theorem \ref{Th:semiperfect-Fnil} and Remarks \ref{Re:semiperf-pdtorsion} and \ref{Rk:Acris-DF}.
\end{proof}

\subsection{Infinite complete intersections}

For a perfect $\FF_p$-algebra $S_0$ and a set $I$ we consider the 
semiperfect ring
\begin{equation}
\label{Eq:RI}
R=S_0[\{Y_i\}_{i\in I}]^{\per}/(\{Y_i\}_{i\in I})
\end{equation}
where $\per$ means perfect hull.
If the set $I$ is finite, then $R$ is a quotient of a
perfect ring by a regular sequence, and the functors
$\Phi_R^{\cris}$ and $\DFfun_{\Spec R}$ 
are equivalences by \cite[Cor.~5.11 \& 5.13]{Lau:Semiperfect}.
In the following we verify that this also holds when $I$ is infinite.

First we construct a lift of $R$, i.e.\ a \chang{$p$-complete}
and $\ZZ_p$-flat ring $A$ such that $A/p=R$ with a Frobenius
lift $\sigma:A\to A$.
Let $S=R^{\flat}=\varprojlim(R,\phi)$
and $R=S/J$ as earlier.
By a slight abuse of notation we write
\[
Y_i=(Y_i^{p^{-r}})_r\in S.
\] 
Let $K_0\subseteq W(S)$ be the ideal 
generated by the elements $[Y_i]$ for $i\in I$,
let $K\subseteq W(S)$ be the closure of $K_0$ 
with respect to the limit topology in
\[
W(S)=\varprojlim_{n,m}W_n(S/\phi^m(J)),
\]
and let $A=W(S)/K$. The Frobenius of $W(S)$ induces $\sigma:A\to A$.

\begin{Lemma}
\label{Le:lift-RI}
The ring $A$ is \chang{$p$-complete} and $\ZZ_p$-flat and $A/p=R$.
\end{Lemma}

\begin{proof}
We will write $A_n=A/p^n$.
If the set $I$ is finite, then $K_0=K$ and the lemma is
easily verified using that the elements $Y_i$ of $S$ form a regular sequence.
In the general case,
by definition we have
\[
A=\varprojlim_{n,m}(W_n(S/\phi^m(J))/\bar K_0)
\]
where $\bar K_0$ is the image of $K_0$.
For fixed $n$, the ring in the limit stabilizes for $m\ge n-1$
because the ideal $\phi^m(J)/\phi^{m+1}(J)$ is generated by all
$Y_i^{p^m}$ for $i\in I$, and for $r\le n-1$ the shift
$v^r([Y_i^{p^m}])=p^r[Y_i^{p^{m-r}}]$ lies in $K_0$.
The stable value is given by
\[
W_n(S/\phi^{n-1}(J))/\bar K_0\cong A_n.
\]
Clearly $A_1=R$. We have to show that $A_n$ is flat over $\ZZ/p^n$.
For each finite subset $M\subseteq I$ let $R_M\subseteq R$
be the analogue of $R$ with $M$ in place of $I$,
let $S_M=R_M^\flat$, and let $J_M$ be the kernel of $S_M\to R_M$,
which is generated by $Y_i$ for $i\in M$,
and let $K_M\subseteq W(S_M)$ be the ideal generated by $[Y_i]$ for $i\in M$.
Then the ring
\[
A_{M,n}=W_n(S_M/\phi^{n-1}(J_M))/\bar K_M=W_n(S_M)/\bar K_M
\]
is flat over $\ZZ/p^n$. 
Now $R$ is the colimit over $M$ of $R_M$,
and $\phi^{-n}$ induces an isomorphism $S/\phi^n(J)\cong R$.
It follows that $A_n$ is the colimit over $M$ of $A_{M,n}$, 
so $A_n$ is flat over $\ZZ/p^n$.
\end{proof}

\begin{Lemma}
\label{Le:Acris-RI}
\chang{For $R$ as in \eqref{Eq:RI},} the ring $A_{\cris}(R)$ is $p$-torsion free.
\end{Lemma}

\begin{proof}
If the set $I$ is finite, the lemma holds because the
elements $Y_i$ of $S$ form a regular sequence.
In general we note that $A_{\cris}(R)/p^n$ is the PD envelope
relative to $\Sigma_n$ of the kernel of $W_n(S)\to R$,
or equivalently of $W_n(S/\phi^n(J))\to R$; 
note that $W_n(\phi^n(J))$ maps to zero 
in every PD thickening of $R$ annihilated by $p^n$
since $v^r([a^{p^n}])=p^r[a]^{p^{n-r}}$ becomes divisible by $p^n$ for $a\in J$. 
Using the notation of the proof of Lemma \ref{Le:lift-RI},
the arrow $W_n(S/\phi^n(J))\to R$ is the colimit over all
finite subsets $M\subseteq I$ of the arrows $W_n(S_M/\phi^n(J_M))\to R_M$.
The PD envelope relative to $\Sigma_n$ of the latter 
is flat over $\ZZ/p^n$, and the lemma follows.
\end{proof}

Lemma \ref{Le:lift-RI} implies that there is a frame
$
\u A=(A,pA,R,\sigma,\sigma_1)
$
with $\sigma_1=p^{-1}\sigma$.
The universal property of $A_{\cris}(R)$ gives a ring homomorphism
$A_{\cris}(R)\to A$, which is a frame homomorphism 
\[
\varkappa:\u A{}_{\cris}(R)\to\u A
\]
since $A$ is torsion free.
We have functors
\[
\BT(\Spec R)
\xrightarrow{\Phi_R^{\cris}}
\Win(\u A{}_{\cris}(R))
\xrightarrow{\varkappa^*}
\Win(\u A)
\]
where $\Phi_R^{\cris}$ can be defined directly by evaluation of the
Dieudonn\'e crystal at $A_{\cris}(R)$ 
since this ring is torsion free.
The composition $\Phi_A=\varkappa^*\circ\Phi_R^{\cris}$
is defined by evaluation of the Dieudonn\'e crystal at $A$.

\begin{Lemma}
\label{Le:varkappa-RI} 
The homomorphism $\varkappa$ induces an equivalence of windows.
\end{Lemma}

\begin{proof}
Let $N\subseteq A_{\cris}(R)$ be the kernel of $\varkappa$.
By Proposition \ref{Pr:deform-win} it suffices to show
that $\sigma$ induces a \chang{locally} nilpotent endomorphism of $N/p$.
Since $A$ is torsion free, 
$N/p$ is the kernel of $A_{\cris}(R)/p\to A/p=R$,
which is generated by the divided powers
$[Y_i]^{[n]}$ for $i\in I$ and $n\ge 1$.
Using \eqref{Eq:sigma1-an} it follows that $\sigma_1^2=0$ on $N/p$.
\end{proof}

Let $\tilde I$ be the kernel of $A\to R\xrightarrow\phi R$.
One verifies that $\tilde I$ is a PD ideal,
using that $\tilde I$ is generated by $p$ and $[Y_i^{1/p}]$ for $i\in I$,
and $[Y_i^{1/p}]^p=0$ in $A$.

\begin{Lemma}
\label{Le:PD-tI}
The divided powers on $\tilde I$ induce \chang{locally} nilpotent
divided powers $\delta$ on the ideal 
$\bar J=\tilde I/pA=\Ker(\phi:R\to R)$.
\end{Lemma}

\begin{proof}
The ideal $\bar J$ is generated by $Y_i^{1/p}$ for $i\in I$.
Since $[Y_i^{1/p}]\in A$ is an inverse image of \chang{$Y_i^{1/p}$} with $[Y_i^{1/p}]^p=0$,
it follows that $\delta_p(Y_i^{1/p})=0$.
\end{proof}

\begin{Thm}
\label{Th:semiperfect-RI}
For $R$ as in \eqref{Eq:RI} the functor $\Phi_R^{\cris}$ is an equivalence.
\end{Thm}

\begin{proof}
By Lemma \ref{Le:PD-tI} the hypotheses of \cite[Th.~5.7]{Lau:Semiperfect}
are satisfied, which implies that the functor 
$\Phi_A=\varkappa^*\circ\Phi_R^{\cris}$ is an equivalence.
The result follows since $\varkappa^*$ is an equivalence by
Lemma \ref{Le:varkappa-RI}.
\end{proof}

\begin{Cor}
\label{Co:semiperfect-RI}
For $R$ as in \eqref{Eq:RI} the functor
$
\DFfun_{\Spec R}
$
is an equivalence.
\end{Cor}

\begin{proof}
\chang{Use Theorem \ref{Th:semiperfect-RI}, Lemma \ref{Le:Acris-RI}, and Remark \ref{Rk:Acris-DF}.}
\end{proof}

\subsection{Extension to rings where $p$ is nilpotent}

Let $R$ be a ring in which $p$ is nilpotent such that $R_0=R/p$ is semiperfect.
There is a universal \chang{$p$-complete} PD thickening $\pi:A_{\cris}(R)\to R$,
which gives a frame $\u A{}_{\cris}(R)$ as in \eqref{Eq:uAcrisR}.
Namely $A_{\cris}(R)=A_{\cris}(R_0)$ as a ring, 
$\pi$ is the unique lift of $A_{\cris}(R_0)\to R_0$,
and $\sigma_1:I_{\cris}(R)\to A_{\cris}(R)$ 
is the restriction of $\sigma_1:I_{\cris}(R_0)\to A_{\cris}(R_0)$.
Windows over $\u A{}_{\cris}(R)$ correspond to windows $\u M$ 
over $\u A{}_{\cris}(R_0)$ together with a lift of 
$\bar M_1\subseteq M\otimes_{A_{\cris}(R_0)}R_0$
to a direct summand of $M\otimes_{A_{\cris}(R_0)}R$.
Hence the Hodge filtration of a $p$-divisible group gives
an extension of $\Phi_{R_0}^{\cris}$ to a contravariant functor
\begin{equation}
\label{Eq:PhiRcris-gen}
\Phi_R^{\cris}:\BT(R)\to\Win(\u A{}_{\cris}(R)).
\end{equation}
If $p\ge 3$ and $\Phi_{R_0}^{\cris}$ is an equivalence,
then $\Phi_R^{\cris}$ is an equivalence.

\section{Dieudonn\'e theory by $p$-root descent}
\label{Se:descent}

In this section we derive properties of the crystalline Dieudonn\'e
functor by descent from the semiperfect case, using a topology that
allows infinite extractions of $p$-th roots.

\subsection{A $p$-root topology}
\label{Subsect:pr}

Let $R$ be a ring in which $p$ is nilpotent.
For a family $(a_i)_{i\in I}$ of elements of $R$ we define 
$R[(a_i^{1/p^n})_{i\in I}]$ and $R[(a_i^{1/p^\infty})_{i\in I}]$
as in \eqref{Eq:Rapn} and \eqref{Eq:Rapinf}.

\begin{Defn}
Let $R$ be a ring in which $p$ is nilpotent.
A morphism $f:\Spec R'\to\Spec R$ is called
\begin{itemize}
\item
an extraction of a $p$-th root if $R'\cong R[a^{1/p}]$ with $a\in R$,
\item
a simultaneous extraction of $p$-th roots
if $R'\cong R[(a_i^{1/p})_{i\in I}]$ for some family $(a_i)_{i\in I}$ of elements of $R$,
\item
a $p$-root extension if $R'=\varinjlim_i B_i$ 
for a sequence of rings $R=B_0\to B_1\to B_2\to\ldots$
where each $\Spec B_{i+1}\to\Spec B_{i}$ 
is a simultanous extraction of $p$-th roots.
\end{itemize}
\end{Defn}

The main example of a $p$-root extension is
$R\to R_\infty=R[(a_i^{1/p^\infty})_{i\in I}]$
for a family $(a_i)$ as above.
If $R/p$ is generated by the $a_i$ as an algebra over $\phi(R/p)$,
then $R_\infty/p$ is semiperfect.
The class of $p$-root extensions is closed under composition
and under base change,
moreover $p$-root extensions are faithfully flat universal homeomorphisms.
We consider the topology which is generated
by the Zariski topology and $p$-root extensions:

\begin{Defn}
\label{Def:top}
Let $X$ be a scheme on which $p$ is locally nilpotent.
A  morphism $f:X'\to X$ 
is called a $p$-root morphism ($\nameoftop$-morphism) 
if $f$ is a $p$-root extension Zariski locally, 
i.e.\ for \chang{any} $y\in X'$ there are affine open
sets $U'\subseteq X'$ and $U\subseteq X$ with $y\in U'$ and $f(U')\subseteq U$
such that $f:U'\to U$ is a $p$-root extension. 
A $\nameoftop$-covering is a surjective family of $\nameoftop$-morphisms.
We denote by $X_\nameoftop$ the category of $\nameoftop$-sheaves on
the category of $X$-schemes.
\end{Defn}

\begin{Remark}
In \cite{Fontaine-Jannsen} a morphism of $\FF_p$-schemes 
is called a $p$-root-morphism or $p$-morphism if Zariski locally 
it is a finite succession of extractions of $p$-th roots.
Thus the $\nameoftop$-morphisms used here a pro-version of the 
$p$-morphisms of \cite{Fontaine-Jannsen}.
For sheaves that commute with filtered colimits of rings, 
passing to the pro-version makes no difference.
\end{Remark}

For a scheme $X$ on which $p$ is locally nilpotent
let $(X/\Sigma)_{\CRIS,\nameoftop}$ be the topos of $\nameoftop$-sheaves
on $\CRIS(X/\Sigma)$, where 
a $\nameoftop$-covering of $(U,T,\delta)\in\CRIS(X/\Sigma)$ 
corresponds to a $\nameoftop$-covering of $T$ as usual.
A quasi-coherent (resp.\ finite locally free) crystal on $X$ 
is a quasi-coherent (resp.\ finite locally free)
$\OOO_{X/\Sigma}$-module on $(X/\Sigma)_{\CRIS,\Zar}$, 
or equivalently on $(X/\Sigma)_{\CRIS,\nameoftop}$;
the equivalence holds by faithfully flat descent of modules.

\begin{Remark} [Functoriality]
\label{Rk:X/S-pr-funct} 
For a morphism $f:X'\to X$ 
there is a morphism of topoi
$f=f_{\CRIS}:(X'/\Sigma)_{\CRIS,\nameoftop}\to(X/\Sigma)_{\CRIS,\nameoftop}$
such that $f_{\CRIS}^{-1}$ is the evident restriction functor,
and $f_{\CRIS}$ becomes a morphism of ringed topoi using the \chang{obvious} identification 
$f_{\CRIS}^{-1}(\OOO_{X/\Sigma})=\OOO_{X'/\Sigma}$;
see \cite[1.1.10]{BBM}.
Similarly, $f$ induces a morphism of ringed topoi 
$f:X'_\nameoftop\to X_\nameoftop$.
\end{Remark}

\begin{Lemma}
\label{Le:descent-crystals}
The fibered categories of quasi-coherent crystals,
of finite locally free crystals,
of Dieudonn\'e crystals, and of filtered Dieudonn\'e crystals 
over the category of schemes on which $p$ is locally nilpotent 
are $\nameoftop$-stacks.
\end{Lemma}

\begin{proof}
This is straightforward, using that for a $p$-root extension
$f:\Spec R'\to\Spec R$ and a PD thickening $\Spec R\to\Spec B$  
there is a $p$-root extension $g:\Spec B'\to\Spec B$ which lifts $f$,
and $\Spec R'\to \Spec B'$ is naturally a PD thickening since $B\to B'$ is flat. 
\chang{For example, if $(X_i\to X)_i$ is a $\nameoftop$-covering and if quasi-coherent crystals $\MMM_i$ over $X_i$ with descent isomorphisms over $X_i\times_XX_j$ are given, one defines a quasi-coherent crystal $\MMM$ over $X$ as follows. For $(U,T,\delta)\in\CRIS(X/\Sigma)$, if the given morphism $U\to X$ factors over some $X_i$, then $\MMM_{T}=(\MMM_i)_{T}$, which is independent of the factorisation by the descent isomorphisms. In general, such a factorisation exists $\nameoftop$-locally in $T$, and $\MMM_{T}$ is defined by fpqc descent of quasi-coherent sheaves.}
\end{proof}

\subsection{Schemes with a local $p$-basis}

\begin{Thm}
\label{Th:eq-p-basis}
\chang{Let $X$ be an $\FF_p$-scheme which locally has a $p$-basis.}
Then the functor
\chang{$\Dfun_X:\BT(X)\to\Dcat(X)$}
from $p$-divisible groups to Dieudonn\'e crystals is an equivalence.
\end{Thm}

\begin{proof}
\chang{We will show that the functor $\DFfun_X:\BT(X)\to\DFcat(X)$ is an equivalence, which is equivalent to the theorem by Lemma \ref{Le:DF-D-eq}.}
We may assume that $X=\Spec R$ where $R$ has a $p$-basis $(x_i)_{i\in I}$.
Let $R'=R^{\per}$ be the perfect hull of $R$.
Then $R'=R[(x_i^{1/p^\infty})_{i\in I}]$,
in particular $\Spec R'\to\Spec R$ is a $p$-root extension.
Let $R''=R'\otimes_RR'$ and $R'''=R'\otimes_RR'\otimes_RR'$.
Let $y_i=x_i\otimes 1-1\otimes x_i\in R'\otimes_{\FF_p} R'$.

\begin{Lemma}
\label{Le:R''-RI}
If $R''$ is viewed as an $R'$-algebra by one of the factors,
there is an isomorphism of $R'$-algebras
\[
R''\cong R'[\{Y_i\}_{i\in I}]^{\per}/(\{Y_i\}_{i\in I})
\]
which sends $Y_i^{p^{-r}}$ to the image of $y_i^{p^{-r}}$ 
under $R'\otimes_{\FF_p} R'\to R''$.
\end{Lemma}

\begin{proof}
The monomials $\u x^{\u a}=\prod x_i^{a_i}$ %(
with exponents $\u a\in ([0,1)\cap\ZZ[1/p])^{(I)}$ %]
form a basis of $R'$ as an $R$-module.
Hence the elements $1\otimes\u x^{\u a}$ form a basis of $R''$
as an $R'$-module by the first factor.
By writing out $\u y^{\u a}=\prod y_i^{a_i}$ as an $R'$-linear combination
of the elements $1\otimes\u x^{\u a}$ one sees that
the elements $\u y^{\u a}$ form an $R'$-basis of $R''$ as well.
Since $y_i$ maps to zero in $R''$, the lemma follows.
\end{proof}

We continue the proof of Theorem \ref{Th:eq-p-basis}.
By Lemma \ref{Le:R''-RI}
the ring $R''$ takes the form \eqref{Eq:RI} with $S_0=R'$,
and the same holds for $R'''=R''\otimes_{R'}R''$.
Hence the functor $\DFfun_{\Spec T}$ is an equivalence 
for $T\in\{R',R'',R'''\}$ by Corollary \ref{Co:semiperfect-RI}.
By faithfully flat descent of $p$-divisible groups and 
$\nameoftop$-descent of filtered Dieudonn\'e crystals \chang{with respect to} 
$\Spec R'\to\Spec R$ (Lemma \ref{Le:descent-crystals})
it follows that the functor $\DFfun_{\Spec R}$ is an equivalence.
\end{proof}

\begin{Remark}
\label{Rk:dJ-case}
Theorem \ref{Th:eq-p-basis} extends
\cite[Main Theorem 1]{Jong:Crystalline}, 
which states that the functor $\Dfun_{\FX}$ is an equivalence 
if $\FX$ is a formal scheme which locally takes the form $\FX=\Spf A$ 
for $A$ as in Remark \ref{Rk:Rings-dJ}.
Indeed, Theorem \ref{Th:eq-p-basis} applies to $X=\Spec A$,
\chang{which locally has a finite $p$-basis by Lemma \ref{Le:F-finite-reg},}
moreover $\BT(X)\cong\BT(\FX)$ and $\Dcat(X)\cong\Dcat(\FX)$
by \cite[Lemma 2.4.4 \& Prop.\ 2.4.8]{Jong:Crystalline}.
\end{Remark}

\subsection{Locally complete intersections}

\begin{Thm}
\label{Th:eq-Ffin-Fnil-tf}
If $X$ is an $F$-finite and $F$-nilpotent scheme over $\FF_p$
which is PD torsion free, then the functor\/
$
\DFfun_X:\BT(X)\to\DFcat(X)
$
from $p$-divisible groups to filtered Dieudonn\'e crystals is an equivalence.
\end{Thm}

See Definitions \ref{Def:F-nilp} and \ref{Def:PD-tor}
for $F$-nilpotent and PD torsion free schemes.

\begin{proof}
We \chang{may} assume that $X=\Spec R$ \chang{is affine}.
Let $a_1,\ldots,a_r$ generate $R$ as an algebra over $\phi_R(R)$,
using that $R$ is $F$-finite, and let 
\begin{equation}
\label{Eq:R'}
R'=R[(a_i^{1/p^\infty})_{1\le i\le r}]
\end{equation}
and $R''=R'\otimes_RR'$, {} $R'''=R'\otimes_RR'\otimes_RR'$.
Each ring $T\in\{R',R'',R'''\}$
is semiperfect, and $F$-nilpotent by Lemma \ref{Le:Fnil-remains}.
\chang{Moreover, $T$ is PD torsion free since this holds for $R$, 
using Corollary \ref{Co:DAR-lci-tf} and a colimit argument.
Hence the functor $\DFfun_{\Spec T}$ is an equivalence by Corollary \ref{Co:semiperfect-Fnil}.}
\chang{As in the proof of Theorem \ref{Th:eq-p-basis}, by faithfully flat descent of $p$-divisible groups and 
$\nameoftop$-descent of filtered Dieudonn\'e crystals \chang{with respect to} 
$\Spec R'\to\Spec R$ (Lemma \ref{Le:descent-crystals})
it follows that the functor $\DFfun_{\Spec R}$ is an equivalence.}
\end{proof}

\begin{Cor}
\label{Co:eq-lci}
If $X$ is an $F$-finite l.c.i.\ scheme over $\FF_p$, 
then \chang{the functor} $\DFfun_X$ is an equivalence.
\end{Cor}

\begin{proof}
Noetherian and $F$-finite implies excellent by
\cite[Th.~2.5]{Kunz:Noetherian}, 
hence $X$ is PD torsion free by 
Proposition \ref{Pr:PD-tor-lci},
moreover noetherian implies $F$-nilpotent.
\chang{Hence Theorem \ref{Th:eq-Ffin-Fnil-tf} applies.}
\end{proof}

\section{Divided Dieudonn\'e crystals}
\label{Se:divided}

Throughout this section
let $X$ be a scheme on which $p$ is locally nilpotent.
Let $X_\nameoftop$ and $(X/\Sigma)_{\CRIS,\nameoftop}$
be defined as in Section \ref{Subsect:pr}.

\begin{Lemma}
\label{Le:iu}
There are morphisms of topoi
\[
X_{\nameoftop}\xrightarrow {\;i\;}(X/\Sigma)_{\CRIS,\nameoftop}\xrightarrow{\;u\;} X_{\nameoftop}
\]
with $u\circ i=\id$ where $i_*=u^{-1}$ is defined by 
$i_*(F)(U,T,\delta)=F(U)$.
\end{Lemma}

\begin{proof}
Let $i^{-1}(G)(U)=G(U,U,0)$ and $u_*(G)(U)=\Gamma((U/\Sigma)_{\CRIS},G)$.
Then $i_*$ and $i^{-1}$ preserve sheaves and $i^{-1}$ is an
exact left adjoint of $i_*$, and $u_*$ is a right adjoint of $u^{-1}$
on the level of presheaves,
moreover $i^{-1}\circ u^{-1}=\id$. 
Using that $p$-root extensions can be lifted to PD thickenings
one verifies that $u_*$ preserves sheaves and that $u^{-1}$ is exact.
\end{proof}

\begin{Remark}
For $n\ge 1$ let $X_n=X\times\Spec\ZZ/p^n$.
There are analogous morphisms of topoi
\[
(X_n)_{\nameoftop}\xrightarrow {\;i_n\;}(X_n/\Sigma_n)_{\CRIS,\nameoftop}
\xrightarrow{\;u_n\;} (X_n)_{\nameoftop}
\]
with $(u_n)_*(G)(U)=\Gamma((U/\Sigma_n)_{\CRIS},G)$.
\end{Remark}

\begin{Remark}
For a morphism of schemes $f:X'\to X$ 
the morphisms $i$ and $u$ of Lemma \ref{Le:iu} 
are compatible with the morphisms $f:X'_{\nameoftop}\to X_{\nameoftop}$
and $f:(X/\Sigma)_{\CRIS,\nameoftop}\to(X/\Sigma)_{\CRIS,\nameoftop}$
of Remark \ref{Rk:X/S-pr-funct}.
The same holds for $i_n$ and $u_n$.
\end{Remark}

\begin{Remark}
If $U=\Spec R$ is an affine $X$-scheme where $R/p$ is semiperfect, 
the category $\CRIS(U/\Sigma)$ has the initial pro-object $\Spf A_{\cris}(R)$,
and 
\[
u_*(G)(U)=\varprojlim_nG(U,\Spec A_{\cris}(R)/p^n,\delta).
\]
This formula determines the sheaf $u_*(G)$
since schemes $U$ of this type form a basis of the $\nameoftop$-topology of $X$.
\end{Remark}

\subsection*{The sheaf \boldmath{$\OOO_X^{\cris}$} and its locally free modules}

Following Fontaine-Messing \cite{Fontaine-Messing} 
we define $\OOO_X^{\cris}=u_*\OOO_{X/\Sigma}$,
which is a ring in $X_{\nameoftop}$.
Then $i$ and $u$ are morphisms of ringed topoi
\[
(X_{\nameoftop},\OOO_X)
\xrightarrow {\;i\;}
((X/\Sigma)_{\CRIS,\nameoftop},\OOO_{X/\Sigma})
\xrightarrow {\;u\;}
(X_{\nameoftop},\OOO_X^{\cris})
\]
in a natural way.
We also consider the torsion version
$\OOO_{X,n}^{\cris}=(u_n)_*\OOO_{X_n/\Sigma_n}$ in $(X_n)_\nameoftop$.
For a morphism $f:X'\to X$ we have 
$f^{-1}\OOO_X^{\cris}=\OOO_{X'}^{\cris}$
and thus an evident morphism of ringed topoi
\[
f:(X'_{\nameoftop},\OOO_{X'}^{\cris})\to
(X_{\nameoftop},\OOO_X^{\cris}).
\]

The canonical PD ideal $\III_{X/\Sigma}\subseteq\OOO_{X/\Sigma}$ 
is the kernel of $\OOO_{X/\Sigma}\to i_*\OOO_X$.
The functor $u_*$ applied to this map gives a homomorphism
$\OOO_X^{\cris}\to\OOO_X$ with kernel $\III_X^{\cris}=u_*\III_{X/\Sigma}$.

\begin{Lemma}
The homomorphism $\OOO_X^{\cris}\to\OOO_X$ is surjective,
the sheaf $\OOO_X^{\cris}$ is \chang{$p$-complete},
i.e.\ $\OOO_X^{\cris}\cong\varprojlim_n\OOO_X^{\cris}/p^n$, 
and we have $\OOO_X^{\cris}/p^n\cong j_{n*}\OOO_{X,n}^{\cris}$
where $j_n:(X_n)_\nameoftop\to X_\nameoftop$ is the natural morphism.
\end{Lemma}

\begin{proof}
One verifies directly that $\OOO_X^{\cris}=\varprojlim_nj_{n*}\OOO_{X,n}^{\cris}$.
If $U=\Spec R$ is an affine $X$-scheme such that $R/p$ is semiperfect,
then $\OOO_X^{\cris}(U)=A_{\cris}(R)$
and $j_{n*}\OOO_{X,n}^{\cris}(U)=A_{\cris}(R)/p^n$. 
Hence $j_{n*}\OOO_{X,n}^{\cris}$ coincides with the presheaf quotient
$\OOO_X^{\cris}/p^n$ on a basis of the topology,
thus $j_{n*}\OOO_{X,n}\cong \OOO_X^{\cris}/p^n$ as sheaves. 
Since $A_{\cris}(R)\to R$ is surjective, the lemma follows.
\end{proof}

\begin{Lemma}
\label{Le:u**}
The functors $u_*$ and $u^*$ between $\OOO_{X/\Sigma}$-modules
and $\OOO_X^{\cris}$-modules induce inverse equivalences
between the categories of finite locally free modules,
$\LF(\OOO_{X/\Sigma})\cong\LF(\OOO_X^{\cris})$.
\end{Lemma}

\begin{proof}
If $\MMM$ is a finite locally free $\OOO_X^{\cris}$-module,
then $u^*\MMM$ is a finite locally free $\OOO_{X/\Sigma}$-module,
moreover $\MMM\to u_*u^*\MMM$ is an isomorphism because this can be verified locally,
and it holds for $\MMM=\OOO_X^{\cris}$ by the definition of $\OOO_X^{\cris}$.
Let $\MMM$ by a locally free $\OOO_{X/\Sigma}$-module.
We claim that the $\OOO_X^{\cris}$-module $u_*\MMM$ 
is finite locally free and that $u^*u_*\MMM\to \MMM$ is an isomorphism.
This can be verified $\nameoftop$-locally, 
so let $X=\Spec R$ such that $R/p$ is semiperfect.
By passing to a $\nameoftop$-covering of $X$ 
we can assume that $\bar M=\MMM(X,X,0)$ is a finite free $R$-module.
For $n\ge 1$ with $p^nR=0$ let $A_n=A_{\cris}(R)/p^n$.
Then $M_n=\MMM(X,\Spec A_n,\delta)$ is a finite projective $A_n$-module
with $M_n\otimes_{A_n}R=\bar M$,
thus $M_n$ is free by Nakayama's Lemma,
and $u_*(\MMM)(X)=\varprojlim_nM_n$ is a finite free 
module over $A_{\cris}(R)=\OOO^{\cris}_X(X)$.
If $X'\to X$ is a morphism and $X'=\Spec R'$ where $R'/p$ is semiperfect,
then an $R$-basis of $\MMM(X,X,0)$ maps to an $R'$-basis of $\MMM(X',X',0)$,
which implies that an $A_{\cris}(R)$-basis $(e_i)$ of $u_*(\MMM)(X)$ maps to an
$A_{\cris}(R')$-basis of $u_*(\MMM)(X')$.
Hence $u_*\MMM$ is free over $\OOO_X^{\cris}$ with basis $(e_i)$,
and the assertion follows.
\end{proof}

\begin{Cor}
\label{Co:Ocris-Acris}
Let $X=\Spec R$ such that $R/p$ is semiperfect.
Then finite locally free $\OOO_X^{\cris}$-modules $\MMM$
are equivalent to finite projective modules $M$ over $A=A_{\cris}(R)$
by the functors $M=\Gamma(X,\MMM)$ and $\MMM=M\otimes_A\OOO_X^{\cris}$.
\end{Cor}

\begin{proof}
Finite locally free $\OOO_{X/\Sigma}$-modules $\MMM$
are equivalent to finite projective $A$-modules $M$
by the functors $M=\varprojlim_n\MMM(X,\Spec A/p^n,\delta)$
and $\MMM=M\otimes_A\OOO_{X/\Sigma}$.
Hence Lemma \ref{Le:u**} gives the result.
\end{proof}

\subsection*{The frame structure on \boldmath{$\OOO_X^{\cris}$}}

The sheaf $\OOO_X^{\cris}$ carries a natural Frobenius lift 
\[
\sigma:\OOO_X^{\cris}\to\OOO_X^{\cris}
\] 
which can be defined as follows.
If $U=\Spec R$ is an affine $X$-scheme where $R/p$ is semiperfect,
then the Frobenius $\phi:R/p\to R/p$ induces an endomorphism
$\sigma$ of $A_{\cris}(R/p)=A_{\cris}(R)=\OOO_X^{\cris}(U)$.
Since such $U$ form a basis of the $\nameoftop$-topology, 
this defines $\sigma$ as an endomorphism of $\OOO_X^{\cris}$.

\begin{Remark}
Without reference to semiperfect rings 
the definition of $\sigma$ goes as follows.
Assume first that $X$ is an $\FF_p$-scheme.
Then for \chang{every} $X$-scheme $U$ the Frobenius morphism $\phi_U:U\to U$ induces
an endomorphism $\sigma=\phi_U^*$ of
$\OOO_X^{\cris}(U)=\Gamma(U/\Sigma,\OOO_{U/\Sigma})$,
thus an endomorphism $\sigma$ of $\OOO_X^{\cris}$.
In general this gives an endomorphism $\sigma$ 
of $\OOO_{X_0}^{\cris}$,
which is extended to $\OOO_X^{\cris}$ using that 
$\OOO_X^{\cris}=i_*\OOO_{X_0}^{\cris}$ for the natural morphism
$i:(X_0)_\nameoftop\to X_\nameoftop$.
\end{Remark}

\begin{Remark}
\label{Rk:phi-sigma}
If $X$ is an $\FF_p$-scheme,
for an $\OOO_X^{\cris}$-module $\MMM$ there is a natural homomorphism
$w:\sigma^*(\MMM)\to\phi^*(\MMM)$ defined as follows 
(here $\sigma^*$ means scalar extension \chang{along} $\sigma$
and $\phi^*$ means inverse image under $\phi$).
For \chang{any} scheme $U\to X$ the Frobenius $\phi_U:U\to U$ induces a
homomorphism 
\[
\phi_U^*:\MMM(U\to X)\to\MMM(U\to X\xrightarrow\phi X),
\]
which gives a $\sigma$-linear map $\MMM\to\phi^*(\MMM)$,
whose linearization is $w$. 
The homomorphism $w$ is an isomorphism if $\MMM$ is quasi-coherent, 
in particular if $\MMM$ is finite locally free.
Indeed, $w$ is an isomorphism for $\MMM=\OOO_X^{\cris}$,
and thus for quasi-coherent $\MMM$ 
since the construction is local 
and the functors $u^*$ and $\phi^*$ are right exact.
\end{Remark}

\begin{Lemma}
\label{Le:sigma1-O}
There is a unique $\sigma$-linear homomorphism
\[
\sigma_1:\III_X^{\cris}\to\OOO_X^{\cris}
\]
with $p\sigma_1=\sigma$ which is functorial in $X$.
If $p\OOO_X=0$ then $\sigma_1(p)=1$.
\end{Lemma}

\begin{proof}
\chang{For $U=\Spec R$ such that $R/p$ is semiperfect,
Corollary \ref{Co:SW} provides a natural $\sigma$-linear map $\sigma_1:I_{\cris}(R/p)\to A_{\cris}(R/p)$ with $p\sigma_1=\sigma$, and thus $\sigma_1:I_{\cris}(R)\to A_{\cris}(R)$ by restriction. If $U$ is an $X$-scheme, this map is a homomorphism 
$\sigma_1:\III_X^{\cris}(U)\to\OOO_X^{\cris}(U)$,
which extends to a homomorphism
$\sigma_1:\III_X^{\cris}\to\OOO_X^{\cris}$
since affine $X$-schemes with semiperfect reduction
form a \chang{basis} of the $\nameoftop$-topology of $X$. Uniqueness and the formula $\sigma_1(p)=1$ follow from the corresponding statements of Corollary \ref{Co:SW}.}
\end{proof}

Hence we have a frame in $X_\nameoftop$ in the sense of Definition \ref{Def:frame-topos},
\[
\u\OOO{}_X^{\cris}=(\OOO_X^{\cris},\III_X^{\cris},\OOO_X,\sigma,\sigma_1).
\]

\begin{Defn}
A divided Dieudonn\'e crystal over $X$ is a window over $\u\OOO{}_X^{\cris}$
as in Definition \ref{Def:frame-topos}.
The category of divided Dieudonn\'e crystals over $X$
will be denoted by
$\DDcat(X)=\Win(\u\OOO_X^{\cris}/X_{\nameoftop})$.
\end{Defn}

\begin{Lemma}
\label{Le:DD-stack}
Divided Dieudonn\'e crystals over schemes $X$ \chang{on} which $p$
is locally nilpotent form a $\nameoftop$-stack.
\end{Lemma}

\begin{proof}
This is immediate from the definition.
\end{proof}

The following is a special case of a more general result for arbitrary 
rings $R$ in Corollary \ref{Co:DD-Win} below, where also a connection appears.

\begin{Prop}
\label{Pr:dd-win-semiperf}
If $X=\Spec R$ where $R/p$ is semiperfect,
taking global sections gives an equivalence 
\chang{
\[
\Gamma_R:\DDcat(\Spec R)\xrightarrow\sim\Win(\u A{}_{\cris}(R))
\]
}
between divided Dieudonn\'e crystals over $X$
and windows over $\u A{}_{\cris}(R)$.
\end{Prop}

\begin{proof}
We have $A_{\cris}(R)=\Gamma(X,\OOO_X^{\cris})$ and $R=\Gamma(X,\OOO_X)$.
Hence the result follows from Lemma \ref{Le:win-win},
using Corollary \ref{Co:Ocris-Acris}
and the equivalence between finite locally free $\OOO_X$-modules
and finite projective $R$-modules. 
\end{proof}

\subsection*{Relation with filtered Dieudonn\'e crystals}

\begin{Prop}
\label{Pr:tfD-tfO}
The scheme $X$
is PD torsion free
(Definition \ref{Def:PD-tor})
iff the ring\/ $\OOO_X^{\cris}(U)$ is torsion free for every
$\nameoftop$-morphism $U\to X$.
\end{Prop}

\begin{proof}
One reduces to the case that $X=\Spec R$ over $\FF_p$.

Assume that $\OOO_X^{\cris}(U)$ is torsion free for every 
$\nameoftop$-morphism $U\to X$.
Let $A\to R$ be a presentation with \chang{$p$-complete} PD envelope $D$ as in \eqref{Eq:DAR},
and define $A^\infty$ and $R^\infty$ 
as in the proof of Proposition \ref{Pr:PD-env-tf}.
We have an injective ring homomorphism
$D\to D\hatotimes_AA^\infty= A_{\cris}(R^\infty)=\OOO_X^{\cris}(\Spec R^\infty)$.
Since $\Spec R^\infty\to\Spec R$ is a $\nameoftop$-morphism,
$\OOO_X^{\cris}(\Spec R^\infty)$ is torsion free, thus $D$ is torsion free.

Now assume that $R$ is PD torsion free and let $U\to X$ be a $\nameoftop$-morphism. We have to show that $\OOO_X^{\cris}(U)$ is torsion free.
By passing to a $\nameoftop$-covering of $U$ we can assume that $U=\Spec R'$
with semiperfect $R'$ and that $U\to X$ factors as $U\to Z\to X$ 
where $Z\to X$ is an open immersion and $U\to Z$ is a $p$-root covering.
Then $R'$ is PD torsion free by Corollary \ref{Co:DAR-lci-tf} 
using a colimit, thus $\OOO_X^{\cris}(U)=A_{\cris}(R')$ is torsion free.
\end{proof}

\begin{Prop}
\label{Pr:tf-DD-DF}
There is a natural functor
\begin{equation*}
\rho_X:\DDcat(X)\to\DFcat(X)
\end{equation*}
from divided Dieudonn\'e crystals to filtered Dieudonn\'e crystals,
which is an equivalence if $X$
is PD torsion free.
\end{Prop}

\begin{proof}
By passing to a $\nameoftop$-covering of $X$
we can assume that $X=\Spec R$ with semiperfect $R/p$,
because divided Dieudonn\'e crystals and filtered Dieu\-donn\'e crystals 
form $\nameoftop$-stacks by Lemmas \ref{Le:descent-crystals}
and \ref{Le:DD-stack}; 
the condition that $X_0$ is PD torsion free passes
to a $\nameoftop$-covering and
means that $A_{\cris}(R)$ is torsion free;
see Proposition \ref{Pr:tfD-tfO}.

By Proposition \ref{Pr:dd-win-semiperf},
a divided Dieudonn\'e crystal over $X$ corresponds to a window
$\u M=(M,M_1,\Phi,\Phi_1)$ over $\u A=\u A{}_{\cris}(R)$.
Let $F:M^\sigma\to M$ be the linearization of $\Phi$.
There is a unique homomorphism $V:M\to M^\sigma$ with 
$V(\Phi(x))=p\otimes x$ and $V(\Phi_1(x))=1\otimes x$,
which implies $VF=p$ and $FV=p$;
see for example \cite[Remark 2.1.4]{Cais-Lau}.
The triple $(M,F,V)$ corresponds to a Dieudonn\'e crystal over $X$.
Let $M_R=M\otimes_{A_{\cris}(R)}R$ and let $\Filone M_R\subseteq M_R$
be the image of $M_1$. 
We want to define $\rho_X(\u M)=(M,F,V,\Filone M_R)$ 
and have to verify that the filtration is admissible
in the sense of \eqref{Eq:adm-fil}.
The Frobenius of $A/p$ induces $\bar\phi:R/p\to A/p$,
and \eqref{Eq:adm-fil} holds iff 
$\bar\phi^*(\Filone M_{R/p})$ is equal to 
$\Ker(F:M^{\sigma}_{A/p}\to M_{A/p})$.
These are two direct summands of $M^{\sigma}_{A/p}$ of the same rank, 
and the inclusion $\subseteq$ holds since $\Phi=p\Phi_1$ on $M_1$.
Equality follows, and $\rho_X$ is defined.

If $A_{\cris}(R)$ is torsion free,
$\rho_X$ is an equivalence by 
\cite[Prop.~2.6.4]{Cais-Lau}.
\end{proof}

\begin{Cor}
\label{Co:DD-DF}
The functor $\rho_X$ is an equivalence if
$X_0=X\times\Spec\FF_p$ is an excellent l.c.i.\ scheme.
\end{Cor}

\begin{proof}
Use Propositions \ref{Pr:PD-tor-lci} and \ref{Pr:tf-DD-DF}.
\end{proof}

\begin{Lemma}[Reduction modulo $p$]
\label{Le:red-DD-DDF}
There is a $2$-cartesian diagram of categories:
\[
\xymatrix@M+0.2em{
\DDcat(X) \ar[r]^{\rho_X} \ar[d] &
\DFcat(X) \ar[d] \\
\DDcat(X_0) \ar[r]^{\rho_{X_0}} &
\DFcat(X_0)
}
\]
\end{Lemma}

\begin{proof}
This holds because lifts \chang{along} both vertical functors correspond to
lifts of the Hodge filtration.
In more detail, by $\nameoftop$-descent we may assume that
$X=\Spec R$ where $R/p$ is semiperfect.
Let $A=A_{\cris}(R)=A_{\cris}(R_0)$.
Then lifts of an $\u A{}_{\cris}(R_0)$-window $\u M$
to an $\u A{}_{\cris}(R)$-window correspond to lifts of
$\bar M_1\subseteq M\otimes_AR_0$ to a direct summand of $M\otimes_AR$,
which also correspond to lifts of $\rho_{X_0}(\u M)$ to $\DFcat(X)$.
\end{proof}

\section{The divided Dieudonn\'e functor}
\label{Se:functor}

\begin{Prop}
\label{Pr:DD-functor}
For \chang{every} scheme $X$ on which $p$ is locally nilpotent
there is a functor
\[
\DDfun_X:\BT(X)\to\DDcat(X)
\] 
from $p$-divisible groups to divided Dieudonn\'e crystals, 
which is compatible with base change in $X$,
with an isomorphism $\DFfun_X\cong\rho_X\circ\DDfun_X$.
\end{Prop}

Here $\DFfun_X$ is the filtered Dieudonn\'e functor of \eqref{Eq:DDFX},
and $\rho_X$ is the functor of Proposition \ref{Pr:tf-DD-DF}.
We call $\DDfun_X$ the divided Dieudonn\'e functor.

\begin{proof}
One can assume that $X=\Spec R$ where $R/p$ is semiperfect
because $p$-divisible groups, divided Dieudonn\'e crystals,
and filtered Dieudonn\'e crystals form $\nameoftop$-stacks by
faithfully flat descent, Lemma \ref{Le:DD-stack}, 
and Lemma \ref{Le:descent-crystals}, respectively.
Then divided Dieudonn\'e crystals over $X$ are equivalent to 
$\u A{}_{\cris}(R)$-windows by Proposition \ref{Pr:dd-win-semiperf},
and the required functor $\DDfun_X$ is \chang{defined by} the functor 
$\Phi_R^{\cris}$ of \eqref{Eq:PhiRcris} and \eqref{Eq:PhiRcris-gen}.
The isomorphism $\DFfun_X\cong\rho_X\circ\DDfun_X$ 
is part of the construction of $\Phi_R^{\cris}$ 
in \cite[Thm.~6.3]{Lau:Semiperfect} when $R$ is an $\FF_p$-algebra,
and it extends to the general case because for $G\in\BT(X)$
with reduction $G_0\in\BT(X_0)$ the lifts
$\DFfun_X(G)$ of $\DFfun_{X_0}(G_0)$ and 
$\Phi_R^{\cris}(G)$ of $\Phi_{R_0}^{\cris}(G_0)$ 
are both given by the Hodge filtration of $G$.
\end{proof}

\begin{Remark}
\chang{If $X=\Spec R$ where $R/p$ is semiperfect, the construction yields a natural isomorphism $\Gamma_R\circ\DDfun_X\cong\Phi_R^{\cris}$.}
\end{Remark}

\begin{Thm}
\label{Th:Ffin-Fnil}
For an $F$-finite and $F$-nilpotent $\FF_p$-scheme $X$
the divided Dieudonn\'e functor\/ $\DDfun_X$ is an equivalence.
\end{Thm}

\begin{proof}
We \chang{may} assume that $X=\Spec R$ \chang{is affine}.
\chang{Then $R$ is $F$-finite, and we define $R'$, $R''$, $R'''$ 
as in the proof of Theorem \ref{Th:eq-Ffin-Fnil-tf}.}
Each $T\in\{R',R'',R'''\}$ 
is semiperfect and $F$-nilpotent by Lemma \ref{Le:Fnil-remains}, 
hence the functor $\Phi_T^{\cris}$ is an equivalence 
by Theorem \ref{Th:semiperfect-Fnil},
and the functor $\DDfun_{\Spec T}$ is an equivalence
by Proposition \ref{Pr:dd-win-semiperf}.
By faithfully flat descent for $p$-divisible groups
and $\nameoftop$-descent for divided Dieudonn\'e crystals
(Lemma \ref{Le:DD-stack}) it follows that $\DDfun_{\Spec R}$
is an equivalence.
\end{proof}

\begin{Cor}
\label{Co:Ffin-Fnil}
Assume that $p\ge 3$.
Then for every scheme $X$ on which $p$ is locally nilpotent 
such that $X_0=X\times\Spec\FF_p$ is $F$-finite and $F$-nilpotent,
the divided Dieudonn\'e functor $\DDfun_X$ is an equivalence.
\end{Cor}

\begin{proof}
By Remark \ref{Rk:red-BT-DDF} and Lemma \ref{Le:red-DD-DDF},
for $p\ge 3$ the diagram
\[
\xymatrix@M+0.2em@C+1em{
\BT(X) \ar[r]^{\DDfun_X} \ar[d] &
\DDcat(X) \ar[d] \\
\BT(X_0) \ar[r]^{\DDfun_{X_0}} &
\DDcat(X_0)
}
\]
is $2$-cartesian, 
and $\DDfun_{X_0}$ is an equivalence by Theorem \ref{Th:Ffin-Fnil}. 
\end{proof}

\begin{Cor}
\label{Co:Isogeny}
Let $X$ be a scheme with $p^r\OOO_X=0$ 
such that $X_0=X\times\Spec\FF_p$ 
is $F$-finite and $F$-nilpotent.
For $p$-divisible groups $G$ and $H$ over $X$ with Dieudonn\'e crystals 
$\Dfun_X(G)=(\MMM,F,V)$ and\/ $\Dfun_X(H)=(\NNN,F,V)$, 
the homomorphism 
\begin{equation}
\label{Eq:HomGHMN}
\Hom(G,H)\to\Hom_F(\MMM,\NNN)
\end{equation}
is injective with cokernel annihilated by $p^r$.
In particular, the functor\/ $\Dfun_X$ is fully faithful up to isogeny.
\end{Cor}

\begin{proof}
We can assume that $X=X_0$ since this does not change the
target of \eqref{Eq:HomGHMN},
and $\Hom(G,H)\to\Hom(G_0,H_0)$ is injective 
with cokernel annihilated by $p^{r-1}$;
see for example \cite[Lemma 1.1.3]{Katz:Serre-Tate}
and \cite[Lemma 3.2]{Lau:Relations}.

Let $\u\MMM'=\DDfun_X(G)$ and $\u\NNN'=\DDfun_X(H)$
be the divided Dieudonn\'e crystals associated to $G$ and $H$.
Under the equivalence of Lemma \ref{Le:u**},
the homomorphism of $\OOO_{X/\Sigma}$-modules $F:\phi^*\MMM\to\MMM$
corresponds to a homomorphism of $\OOO_X^{\cris}$-modules 
$F':\phi^*\MMM'\to\MMM'$, 
and the composition of $F'$ with the isomorphism 
$w:\sigma^*\MMM'\cong\phi^*\MMM'$ 
of Remark \ref{Rk:phi-sigma} is the homomorphism 
$\Phi:\sigma^*\MMM'\to\MMM'$ which is part of $\u\MMM'$.
Thus \eqref{Eq:HomGHMN} factors as
\[
\Hom(G,H)\to 
\Hom(\u\MMM',\u\NNN')\to 
\Hom_\Phi(\MMM',\NNN')\xrightarrow\sim
\Hom_F(\MMM,\NNN).
\]
Here the first arrow is an isomorphism by Theorem \ref{Th:Ffin-Fnil},
and the second arrow is injective 
with cokernel annihilated by $p$ by Lemma \ref{Le:win-phi-mod}.
\end{proof}

\begin{Remark}
\chang{Theorem \ref{Th:eq-Ffin-Fnil-tf} is a special case of
Theorem \ref{Th:Ffin-Fnil} by Proposition \ref{Pr:tf-DD-DF}, and the proofs of are parallel, interchanging $\DFcat(X)$ and $\DDcat(X)$.
This duplication emphasizes that the results of \S\ref{Se:descent} are independent of the notions of \S\ref{Se:divided}, at the cost of a slightly less economical presentation.}
\end{Remark}

\begin{Remark}
\label{Re:loc-pbasis-DDfun}
\chang{Theorem \ref{Th:eq-p-basis} implies that the functor $\DDfun_X$ is an equivalence for every $\FF_p$-scheme $X$ which locally has a $p$-basis, using Lemma \ref{Le:DF-D-eq}, Remark \ref{Re:pbasis-pd-tf}, and Proposition \ref{Pr:tf-DD-DF}.}
\end{Remark}

\section{Explicit divided Dieudonn\'e crystals}

\label{Se:Explicit-DD}

Let $R$ be a ring in which $p$ is nilpotent.
In this section we describe divided Dieudonn\'e crystals
over $\Spec R$ by windows with a connection.
The procedure is straightforward.
First we construct the relevant frame,
then define windows with a connection,
translate the connection to an HPD stratification,
and finally relate this with divided Dieudonn\'e crystals.

\subsection*{Windows with a connection}

\chang{In the following let $R=A/I$ where $A$ is a $p$-complete and $\ZZ_p$-flat ring with a Frobenius lift $\sigma:A\to A$ such that $A/p$ has a $p$-basis. Let $D=D_{\gamma}(A\to R)^\wedge$ be the $p$-complete PD envelope relative to $\gamma$ and $\bar I\subseteq D$ the natural PD ideal. The Frobenius lift $\sigma$ corresponds to a $\delta$-structure on $A$, and Lemma \ref{Le:delta-pd-sigma1} provides a natural $\sigma$-linear map $\sigma_1:\bar I\to A$ that gives a frame
\[
\u D=(D,\bar I,R,\sigma,\sigma_1)
\]
functorially associated to $(A\to R,\sigma)$.}
Let $(x_i)$ in $A$ map to a $p$-basis of $A/p$.
The module of continuous differentials $\hat\Omega_A$ is topologically
free with basis $(dx_i)$. The derivation $d:A\to\hat\Omega_A$ extends
uniquely to a PD derivation
\[
d:D\to\hat\Omega_D:=D\hatotimes_A\hat\Omega_A,
\]
which means that $d(a^{[n]})=a^{[n-1]}\otimes da$ 
for $a\in\bar I$ and $n\ge 1$,
and this is the universal PD derivation of $(D,\bar I)$ relative to $(\ZZ_p,\gamma)$ by \cite[\href{https://stacks.math.columbia.edu/tag/07HW}{Tag 07HW}]{Stacks-Project}.
Since $\sigma$ is a Frobenius lift, 
the endomorphism $d\sigma$ of $\hat\Omega_A$ 
is divisible by $p$, i.e.\ there is a well-defined $\sigma$-linear map
$(d\sigma)_1:\hat\Omega_A\to\hat\Omega_A$ with $d\sigma=p(d\sigma)_1$.
It induces a $\sigma_D$-linear map $(d\sigma)_1:\hat\Omega_D\to\hat\Omega_D$.

\begin{Defn}
\label{Def:win-nabla}
A connection on a $\u D$-window $\u M=(M, M_1,\Phi,\Phi_1)$
is a connection 
$\nabla:M\to M\otimes_D\hat\Omega_D$ such that $\Phi$ and $\Phi_1$
are horizontal in the sense that the following diagrams commute.
\begin{equation}
\label{Eq:Win-nabla-Phi}
\xymatrix@M+0.2em{
M \ar[r]^-\nabla \ar[d]_\Phi &
M\otimes_D\hat\Omega_D \ar[d]^{\Phi\otimes d\sigma} \\
M \ar[r]^-\nabla & M\otimes_D\hat\Omega_D
}
\end{equation}
\begin{equation}
\label{Eq:Win-nabla-Phi1}
\xymatrix@M+0.2em{
 M_1 \ar[r]^-\nabla \ar[d]_{\Phi_1} &
M\otimes_D\hat\Omega_D \ar[d]^{\Phi\otimes (d\sigma)_1} \\
M \ar[r]^-\nabla & M\otimes_D\hat\Omega_D
}
\end{equation}
We denote by $\Win(\u D)^\nabla$ the category of windows with a connection.
\end{Defn}

\begin{Remark}
If $D$ is torsion free, \eqref{Eq:Win-nabla-Phi} implies
\eqref{Eq:Win-nabla-Phi1}.
If $R$ is an $\FF_p$-algebra,
\eqref{Eq:Win-nabla-Phi1} implies \eqref{Eq:Win-nabla-Phi}
because for $x\in M$ we have $px\in M_1$ and $\Phi(x)=\Phi_1(px)$.
\end{Remark}

\begin{Lemma}
A connection $\nabla$ on a $\u D$-window $\u M$ is necessarily
integrable and topologically quasi-nilpotent.
\end{Lemma}

\begin{proof}
Let $G=\nabla\circ\nabla:M\to M\otimes_D\Lambda^2\hat\Omega_D$,
where $\Lambda^2$ is taken in the topological sense.
Then $G$ is $D$-linear and the following diagrams commute.
\[
\xymatrix@M+0.2em{
 M_1 \ar[r]^-G \ar[d]_{\Phi_1} &
M\hatotimes_D\Lambda^2\hat\Omega_D 
\ar[d]^{\Phi\otimes p\Lambda^2(d\sigma)_1} \\
M \ar[r]^-G &
M\hatotimes_D\Lambda^2\hat\Omega_D
}
\qquad
\xymatrix@M+0.2em{
M \ar[r]^-G \ar[d]_{\Phi} &
M\otimes_D\Lambda^2\hat\Omega_D 
\ar[d]^{\Phi\otimes p^2\Lambda^2(d\sigma)_1} \\
M \ar[r]^-G &
M\otimes_D\Lambda^2\hat\Omega_D
}
\]
Since $M$ is generated by the images of $\Phi_1$ and $\Phi$ as a $D$-module
we deduce: if $G(M)$ lies in 
$p^rM\hatotimes_D\Lambda^2\hat\Omega_D$ for some $r\ge 0$
then the same holds for $r+1$. 
Hence $G=0$ and $\nabla$ is integrable.
Let $\theta_i:\hat\Omega_D\to D$ be given by $dx_i\mapsto 1$
and $dx_j\mapsto 0$ for $i\ne j$.

Let $N_i:M\to M$ be $\nabla$ composed with ${\id_M}\otimes\theta_i$.
Then the $N_i$ commute, and for \chang{any} $x\in M$ and $r\ge 1$,
almost all $N_i(x)$ are zero mod $p^r$.
We have to show that for $x\in M$ and \chang{every} $i$
some power $N_i^m(x)$ is zero in $M/pM$.
It suffices to consider $x\in\Phi_1(M)$ or $x\in\Phi(M)$
since these generate $M$.
If $x\in\Phi_1(M)$ then $\nabla(x)$ lies in the image of 
$D\Phi(M)\otimes\hat\Omega_D$ by \eqref{Eq:Win-nabla-Phi1},
thus $N_i(x)\in D\Phi(M)$.
Since the derivation $N_i:D\to D$ is nilpotent modulo $p$,
it suffices to consider $x\in\Phi(M)$.
Then $\nabla(x)=0$ mod $p$ by \eqref{Eq:Win-nabla-Phi}.
\end{proof}

\subsection*{Windows with an HPD stratification}

We keep the assumptions on $(A\to R,\sigma)$.
The definition of an HPD stratification
on a $\u D$-window is straightforward:
For $m\ge 0$ let 
\[
A(m)=A\hatotimes_{\ZZ_p}\ldots\hatotimes_{\ZZ_p} A
\]
with $m+1$ factors and let 
\[
D(m)=D_\gamma(A(m)\to R)^\wedge.
\]
We define $\sigma:A(m)\to A(m)$ by $\sigma$ on the factors. 
Lemma \chang{\ref{Le:delta-pd-sigma1}} gives frames 
\[
\u D(m)=(D(m),\bar I(m),\sigma,\sigma_1)
\] 
for $m\ge 0$, 
which form a cosimplicial frame $\u D(*)$ with $\u D(0)=\u D$.

Let $p_0,p_1:\u D\to\u D(1)$ be the homomorphisms that
correspond to the first and second coordinate,
and define $q_i:\u D\to\u D(2)$ for $0\le i\le 2$
and $q_{ij}:\u D(1)\to\u D(2)$ for $0\le i<j\le 2$ similarly.

\begin{Defn}
An HPD stratification on a $\u D$-window $\u M$ 
is an isomorphism of $\u D(1)$-windows
$\varepsilon:p_0^*\u M\cong p_1^*\u M$ such that
$q_{02}^*\varepsilon=q_{12}^*\varepsilon\circ q_{01}^*\varepsilon$
over $\u D(2)$.
Let $\Win(\u S)^{\HPD}$ be the category of $\u D$-windows with an
HPD stratification.
\end{Defn}

\begin{Prop}
\label{Pr:nabla-HPD}
There is an equivalence $\Win(\u D)^\nabla\cong\Win(\u D)^{\HPD}$.
\end{Prop}

The proof is standard, 
but some care is required because the base change of windows is
not \chang{given by} the tensor product in all components.
Let us first recall the explicit description of the rings $D(m)$.
Let 
\[
E(m)=D_\gamma(A(m)\to A)^\wedge.
\]
Since $E(m)$ is an augmented PD $A$-algebra,
\cite[I Cor.~1.7.2]{Berthelot:CohCristalline}
implies that 
\begin{equation}
\label{Eq:EAE}
E(m)\hatotimes_AE(n)\cong E(n+m)
\end{equation}
with respect to $A\to E(m)$ by the last factor and $A\to E(n)$
by the first factor, and
\begin{equation}
\label{Eq:EAD}
E(m)\hat\otimes_AD\cong D(m)
\end{equation}
with respect to $A\to E(m)$ by any of the factors;
this is also a consequence of 
\cite[IV Cor.~1.3.5]{Berthelot:CohCristalline}.
It follows that
\begin{equation}
\label{Eq:DDD}
D(m)\hatotimes_DD(n)\cong D(m+n).
\end{equation}
Recall that $(x_i)$ in $A$ map to a $p$-basis of $A/p$.
Let $\xi_i=x_i\otimes 1-1\otimes x_i\in A(1)$.

\begin{Lemma}
\label{Le:E(1)-xi}
As an $A$-algebra by \chang{any} of the \chang{two} factors,
$E(1)$ is the \chang{$p$-completion}
of the PD polynomial algebra $A\langle\xi_i\rangle^\wedge$ 
in the variables $\xi_i$.
\end{Lemma}

See \cite[Cor.~1.3.2]{Berthelot-Messing}.
We give a direct proof for completeness.

\begin{proof}
\chang{Let $B=A[\xi_i]$ as a polynomial ring and let $E'=D_\gamma(B\to A)^\wedge$,
which is the ring $A\langle\xi_i\rangle^\wedge$ of the lemma. The natural homomorphism $B\to A(1)$ given by $a\mapsto a\otimes 1$ on $A$ induces a PD homomorphism $E'\to E(1)$, which will be shown to be isomorphism mod $p^m$ for all $m$. 
We can temporarily change the Frobenius lift $\sigma$ on $A$ and assume that $\sigma(x_i)=x_i^p$, using \cite[Prop.~1.2.6]{Berthelot-Messing} or \cite[Lemma 1.2.2]{Jong:Crystalline} and its proof. 
The resulting Frobenius lift $\sigma$ on $A(1)$ stabilises $B$. 
Since the $x_i$ form a $p$-basis of $A/p$ and the $x_i,\xi_i$ form a $p$-basis
of $A(1)/p$, the following diagram is cocartesian.
\[
\xymatrix@M+0.2em{
B/p^m \ar[r] \ar[d]_{\sigma^m} & 
A(1)/p^m \ar[d]^{\sigma^m} \\
B/p^m \ar[r] & A(1)/p^m
}
\]
Let $K$ be the kernel of $B\to A$ and let $K_m$ be generated by all $p^ib^{p^{m-i}}$ for $b\in K$ and $0\le i\le m$. Then $\sigma(K_i)\subseteq K_{i+1}$ and hence $\sigma^m(K)\subseteq K_m$.
If $S\to A/p^m$ is a PD thickening of $\ZZ/p^m$-algebras and
$f:B/p^m\to S$ is a homomorphism of thickenings of $A/p^m$, it follows that $f\circ\sigma^m$ annihilates the image of $K$ and thus factors as $B/p^m\to A/p^m\to S$. The same holds for $A(1)$ in place of $B$.
It follows that $f$ extends to a unique homomorphism $A(1)/p^m\to S$, and hence $E'/p^m\to E(1)/p^m$ is bijective as required.}
\end{proof}

For $1\le j\le m$ let $\xi_{ij}=q_{j-1,j}(\xi_i)$ in $A(m)$.
Lemma \ref{Le:E(1)-xi} 
\chang{together with \eqref{Eq:EAE} and \eqref{Eq:EAD}}
give:

\begin{Lemma}
$D(m)$ viewed as a $D$-algebra by any of the
factors is the \chang{$p$-completion} of the free PD polynomial algebra
$D\langle\xi_{ij}\rangle^\wedge$.
\qed
\end{Lemma}

We consider the following truncations of the PD envelopes
$E(m)$ and $D(m)$.
Let $J$ be the kernel of $A(1)\to A$,
let $\bar J$ be the kernel of $E(1)\to A$, 
and let $\bar J^{[m]}\subseteq E(1)$
be the $p$-adic closure of the $m$-th PD power of $\bar J$.
For $r\ge 1$ we set
\[
E(1)_r=E(1)/\bar J^{[r]},
\qquad 
D(1)_r=D(1)\hatotimes_{E(1)}E(1)_r,
\]
and
\[
D(m)_r=D(1)_r\hatotimes_D\ldots\hatotimes_DD(1)_r
\]
as a quotient of $D(m)=D(1)\hatotimes_D\ldots\hatotimes_DD(1)$.
Then \eqref{Eq:EAD} implies that 
$D(1)_r$ is isomorphic to $D\hatotimes_EE(1)_r$ 
with respect to the homomorphisms $E\to E(1)_r$ by either of the two factors.
\chang{In particular, we have $D(1)_2\cong D\oplus\hat\Omega_D$, and $D(m)_2$ contains $(\hat\Omega_D)^{\hat\otimes m}$ as an ideal.}

\begin{Lemma}
For $m,r\ge 1$ there is a well-defined frame 
\[
\u D(m)_r=(D(m)_r,\bar I(m)_r,R,\sigma,\sigma_1)
\]
as a quotient of $\u D(m)$.
\end{Lemma}

\begin{proof}
Let $N$ be the kernel of $D(m)\to D(m)_r$. 
We have to show that $\sigma_1(N)\subseteq N$.
For the ring $R'=A/p$ in place of $R$
we get an analogous frame $\u D'(m)=\u E(m)$
and an ideal $N'=\Ker(E(m)\to E(m)_r)$. 
There is a frame homomorphism $\u E(m)\to\u D(m)$,
and $N$ is topologically generated by the image of $N'$.
Hence it suffices to show that $\sigma_1$ preserves $N'$.
But $E(m)_r$ is a topologically free $A$-module, thus $\ZZ_p$-flat, 
hence $N'\cap pE(m)=pN'$. Since $\sigma$ preserves $N'$
it follows that $\sigma_1$ preserves $N'$.
\end{proof}

\begin{Lemma}
\label{Le:D1mDm2}
For $m\ge 1$ 
the homomorphism \chang{$q_{0m}:D(1)\to D(m)$} induces an injective homomorphism
\chang{$\bar q_m:D(1)_{m+1}\to D(m)_2$}.
\end{Lemma}

\begin{proof}
\chang{For $r\le m$ let $s_r(T_1,\ldots,T_m)$ be the $r$-th elementary symmetric polynomial, and let $s_r(T_1,\ldots,T_m)=0$ for $r>m$. 
We have 
\[
q_{0m}(\xi_i)=\xi_{i1}+\xi_{i2}+\ldots+\xi_{im},
\]
so the image of $q_{0m}(\xi_i^{[r]})=q_{0m}(\xi_i)^{[r]}$ in $D(m)_2$ 
is equal to $s_r(\xi_{i1},\ldots,\xi_{im})$.
Similarly, 
for a PD monomial $a=\prod\xi_i^{[r_i]}$ 
of total degree $r=\sum r_i$ the image of
$q_{0m}(a)$ in $D(m)_2$ is equal to $\prod s_{r_i}(\xi_{i1},\ldots,\xi_{im})$,
which is zero if $r>m$. 
It follows that $q_{0m}$ induces a homomorphism $\bar q_m$ as indicated.
For $m\ge 2$ there is a commutative diagram where the right vertical arrow omits one of the coordinates.
\[
\xymatrix@M+0.2em{
D(1)_{m+1} \ar[d]_\pi \ar[r]^-{\bar q_m} & D(m)_2 \ar[d] \\
D(1)_{m} \ar[r]^-{\bar q_{m-1}} & D(m-1)_2.
}
\]
The kernel of $\pi$ is the free $D$-module with basis all PD monomials $a$ as above of total degree $r=m$. One verifies that the images of these elements are $D$-linearly independent in $(\hat\Omega_D)^{\hat\otimes m}$. Then $\bar q_m$ is injective by induction.}
\end{proof}

\begin{proof}[Proof of Proposition \ref{Pr:nabla-HPD}]
Let $\u M=(M, M_1,\Phi,\Phi_1)$ be a $\u D$-window.
An HPD stratification $\varepsilon$ on the $D$-module $M$ is equivalent to
an integrable and topologically quasi-nilpotent connection $\nabla$ on $M$,
and $\Phi$ commutes with $\varepsilon$ iff $\Phi$ is horizontal \chang{with respect to}
$\nabla$ as in \eqref{Eq:Win-nabla-Phi}. 
We have to show that $\varepsilon$ is a window isomorphism, 
i.e.\ it commutes with $\Phi$ and $\Phi_1$,
iff $\Phi$ and $\Phi_1$ are horizontal \chang{with respect to} $\nabla$ in the sense that
\eqref{Eq:Win-nabla-Phi} and \eqref{Eq:Win-nabla-Phi1} commute.

Let $\bar p_i:\u D\to\u D(1)_2$ be the reduction of $p_i:D\to D(1)$ for $i=0,1$.
The isomorphism $\varepsilon:p_0^*M\cong p_1^*M$ 
of $D(1)$-modules induces an isomorphism
$\bar\varepsilon:\bar p_0^*M\cong\bar p_1^*M$ of $D(1)_2$-modules.
We claim that $\varepsilon$ is a window isomorphism iff this holds for 
$\bar\varepsilon$. Assume the latter holds.
The intersection of the kernels of $D(1)\to D(1)_{r+1}$ 
for varying $r$ is zero.
Hence it suffices to show that the reduction $\bar\varepsilon_{r+1}$ over $D(1)_{r+1}$
of $\varepsilon$ is a window isomorphism over $\u D(1)_{r+1}$.
Let $q=q_{0r}:D(1)\to D(r)$,
which induces $\bar q:D(1)_{r+1}\to D(r)_2$ by Lemma \ref{Le:D1mDm2}.
Then $q^*\varepsilon$ is the composition of all $q_{i(i+1)}^*\varepsilon$,
and the reduction $\bar q^*\bar\varepsilon_{r+1}$
is the composition of all $q_{ij}^*\bar\varepsilon$.
Since $\bar\varepsilon$ is a window isomorphism and since $\bar q$ is
injective it follows that $\bar\varepsilon_{r+1}$ is a window isomorphism
as required.

The homomorphism $\bar\varepsilon$ 
corresponds to a $\bar p_0$-linear map $\bar\varepsilon':M\to\bar p_1^*M$,
and $\bar\varepsilon$ is an isomorphism of $\u D(1)_2$-windows iff
$\bar\varepsilon'$ is a homomorphism of windows relative to 
$p_0:\u D\to\u D(1)_2$.
Let us make this condition explicit.
Let $K$ be the kernel of $D(1)_2\to D$,
thus $K\cong\hat\Omega_D$.
Then $D(1)_2=D\oplus K$ 
\chang{such that $\bar p_i:D\to D(1)$ is given by $\bar p_1(a)=(a,0)$ and}
$\bar p_0(a)=(a,da)$ with $da=a\otimes1-1\otimes a\in K$.
The $\u D(1)_2$-window $\bar p_1^*\u M$ consists of the modules
\chang
{\[
\bar p_1^*M=M\oplus (M\otimes_{D}K),
\]
\[
(\bar p_1^*M)_1= M_1\oplus (M\otimes_{D}K),
\]}
and the homomorphism $\Phi_1:(\bar p_1^*M)_1\to\bar p_1^*M$
is given by $\Phi_1=\Phi_1\oplus(\Phi\otimes\sigma_1)$,
while $\Phi=\Phi\oplus(\Phi\otimes\sigma)$.
Moreover $\bar\varepsilon'(x)=x+\nabla(x)$ 
for $x\in M$ under the identification $K\cong\hat\Omega_D$.
Under this identification, 
\chang{$\sigma_1:K\to K$ corresponds to
$(d\sigma)_1:\hat\Omega_D\to\hat\Omega_D$.}
Indeed, this is clear when $D(1)_2$ is torsion free,
for example for $E(1)_2$ in place of $D(1)_2$,
and the general case follows.
As a consequence, $\bar\varepsilon$ is a window isomorphism
iff $\bar\varepsilon'$ is a window homomorphism
iff $\Phi$ and $\Phi_1$ are horizontal \chang{with respect to} $\nabla$.
\end{proof}

\subsection*{HPD stratifications and divided Dieudonn\'e crystals}

We keep the assumptions on $(A\to R,\sigma)$. Let $X=\Spec R$.

\begin{Prop}
\label{Pr:DD-HPD}
There is an equivalence $\DDcat(X)\cong\Win(\u D)^{\HPD}$
between divided Dieudonn\'e crystals over $X$ and
windows over $\u D$ with an HPD stratification.
\end{Prop}

\begin{Remark}
This extends the usual equivalence between 
finite locally free $\OOO_{X/\Sigma}$-modules
and finite projective $D$-modules with an HPD stratification.
More precisely, 
\chang{let $\u\MMM\in\DDcat(X)$ correspond to 
$(\u M,\varepsilon)\in\Win(\u D)^{\HPD}$. The}
underlying $\OOO_X^{\cris}$-module $\MMM$ is equivalent to the
$\OOO_{X/\Sigma}$-module $u^*\MMM$ by Lemma \ref{Le:u**},
and $M$ is the value of $u^*\MMM$ at the \chang{$p$-complete} PD thickening $D\to R$,
equipped with the canonical HPD stratification.
A technical complication 
\chang{arises from using the 
$\nameoftop$-topology instead of the Zariski topology, 
since etale morphisms lift uniquely over PD thickenings,
but $\nameoftop$-morphisms lift only non-uniquely.}
\end{Remark}

\chang
{\begin{Lemma}
\label{Le:BW(S)}
Let $(B,\sigma)$ be a flat $\ZZ_p$-algebra with a Frobenius lift and let $S$ be a ring. Every ring homomorphism $B\to S$ lifts to a natural homomorphism $B\to W(S)$ that commutes with $\sigma$.
\end{Lemma}

\begin{proof}
This follows from the fact that the forgetful functor from $\delta$-rings to rings has a right adjoint given by $R\mapsto W(R)$, following \cite{Joyal:delta}. Explicitly, \cite[VII Prop.~4.12]{Lazard} provides a homomorphism $B\to W(B)$ which
is a section of $W(B)\to B$ and commutes with $\sigma$, and we take 
$B\to W(B)\to W(S)$.
\end{proof}}

\begin{proof}[Proof of Proposition \ref{Pr:DD-HPD}]
One reduces to the case that $R$ is an $\FF_p$-algebra because
on both sides, lifts from $R_0$ to $R$ correspond to lifts of
the Hodge filtration.
First, we construct a functor
\[
F:\Win(\u D)^{\HPD}\to \DDcat(X).
\]

Let $(\u M,\varepsilon)\in\Win(\u D)^{\HPD}$.
Let $R\to R'$ be a ring homomorphism with semiperfect $R'$.
Assume that $S\to R'$ is a surjective ring homomorphism where $S$ is perfect,
and assume that $g_0:A/p\to S$ lifts $A/p\to R\to R'$.
Such pairs $(S,g_0)$ exist, 
for example $S=R'^\flat=\varprojlim(R',\phi)$ 
allows a lift $g_0:A/p\to S$ using that $A/p$ has a $p$-basis.
If $(S,g_0)$ is given, $g_0$ lifts to a \chang{natural} homomorphism $g_1:A\to W(S)$ 
that commutes with $\sigma$ by Lemma \ref{Le:BW(S)},
and $g_1$ induces a frame homomorphism
\[
g:\u D\to\u A{}_{\cris}(R')
\]
by Lemmas \chang{\ref{Le:delta-pd-sigma1}} and \ref{Le:D-Acris}.
We define an $\u A{}_{\cris}(R')$-window $\u M{}_{R'}$ 
by $\u M{}_{R'}=g^*\u M$, and verify that this is 
independent of $(S,g_0)$ by the usual argument:
Let $(S',g_0')$ be another choice of $(S,g_0)$,
and $g':\u D\to\u A{}_{\cris}(R')$ 
the corresponding frame homomorphism.
Then 
\[
h_0=g_0\otimes g_0':A(1)/p\to S\otimes S_1
\]
gives a homomorphism of frames
\[
h:\u D(1)\to\u A{}_{\cris}(R')
\]
with $h\circ p_0=g$ and $h\circ p_1=g'$. 
The isomorphism $\varepsilon:p_0^*\u M\cong p_1^*\u M$ induces
an isomorphism $g^*\u M\cong g'^*\u M$,
which proves independence of $(S,g_0)$.

The construction of $\u M{}_{R'}$ is functorial in $R'$.
By Proposition \ref{Pr:dd-win-semiperf}, $\u M{}_{R'}$ 
corresponds to a divided Dieudonn\'e crystal $\u\MMM{}_{X'}$ 
over $X'=\Spec R'$.
Since affine semiperfect $X$-schemes 
form a \chang{basis} of the $\nameoftop$-topology of $X$, 
the system $(\u\MMM{}_{X'})_{X'}$
descends to a unique divided Dieudonn\'e crystal $\u\MMM$ over $X$,
and the functor $F$ is defined by
\[
F(\u M,\varepsilon)=\u\MMM.
\]

Next we will define a functor
\[
G:\DDcat(X)\to\Win(\u D)^{\HPD}.
\]
We need a number of rings in order to use descent.
Let $A^{(1)}$ be the \chang{$p$-completion} of $\varinjlim(A,\sigma)$.
Then $A^{(1)}/p$ is the perfection of $A/p$, and $A^{(1)}=W(A^{(1)}/p)$.
Let $R^{(1)}=R\otimes_AA^{(1)}$.
Since $A$ has a $p$-basis, the morphisms
$A/p^r\to A^{(1)}/p^r$ and $R\to R^{(1)}$ are
faithfully flat.
For $n\ge 0$ let
\[
A^{(n)}=A^{(1)}\hatotimes_A\ldots\hatotimes_AA^{(1)}
\]
with $n$ factors,
$R^{(n)}=R\otimes_AA^{(n)}$, and 
\[
D^{(n)}=D_\gamma(A^{(n)}\to R^{(n)})^\wedge.
\]
Then $D^{(n)}=D\hatotimes_AA^{(n)}$
since $A/p^r\to A^{(n)}/p^r$ is flat,
hence
\[
D^{(n)}=D^{(1)}\hatotimes_D\ldots\hatotimes_DD^{(1)}
\]
with $n$ factors, and similarly
\[
R^{(n)}=R^{(1)}\otimes_R\ldots\otimes_RR^{(1)}
\]
with $n$ factors.
As a variant, for $n\ge 1$ we also consider the rings
\[
\tilde A^{(n)}=A^{(1)}\hatotimes_{\ZZ_p}\ldots\hatotimes_{\ZZ_p}A^{(1)}
\]
with $n$ factors.
The ring  $\tilde A^{(n)}/p$ is perfect, hence
\[
A_{\cris}(R^{(n)})=D_\gamma(\tilde A^{(n)}\to R^{(n)})^\wedge
\]
by Lemma \ref{Le:D-Acris}.
The rings $\tilde A^{(n)}$ and $A^{(n)}$ carry a natural Frobenius lift
$\sigma$ induced by $\sigma:A\to A$.
The projection $\tilde A^{(n)}\to A^{(n)}$ commutes with $\sigma$
and thus extends to a frame homomorphism
\[
h^{(n)}:\u A{}_{\cris}(R^{(n)})\to\u D^{(n)}
\]
by Lemma \chang{\ref{Le:delta-pd-sigma1}}.
Thus we have the following commutative diagram of frames,
where $h^{(1)}$ is an isomorphism since $\tilde A^{(1)}=A^{(1)}$.
\[
\xymatrix@M+0.2em{
&
\u A{}_{\cris}(R^{(1)}) \ar@<0.5ex>[r] \ar@<-0.5ex>[r] 
\ar[d]_\cong^{h^{(1)}} &
\u A{}_{\cris}(R^{(2)}) \ar@<1ex>[r] \ar[r] \ar@<-1ex>[r] 
\ar[d]^{h^{(2)}} & 
\u A{}_{\cris}(R^{(3)}) 
\ar[d]^{h^{(3)}} \\
\u D = \u D^{(0)}  \ar[r] &
\u D^{(1)} \ar@<0.5ex>[r] \ar@<-0.5ex>[r] &
\u D^{(2)} \ar@<1ex>[r] \ar[r] \ar@<-1ex>[r] & 
\u D^{(3)} 
}
\]
Each homomorphism $\alpha:\u D^{(n)}\to\u D^{(n+1)}$ 
that appears in the lower row of this diagram induces an isomorphism
\[
\u D^{(n)}\otimes_{D^{(n)}}D^{(n+1)}\xrightarrow\sim\u D^{(n+1)}
\]
where $\otimes$ is taken component-wise, and consequently
the base change of windows \chang{along} $\alpha$ is given by
$\alpha^*\u M=\u M\otimes_{D^{(n)}}D^{(n+1)}$.

Now let $\u\MMM\in\DDcat(X)$.
For $n\ge 1$, the base change of $\u\MMM$ to $\Spec R^{(n)}$ 
corresponds to an $\u A{}_{\cris}(R^{(n)})$-window $\u M{}_{R^{(n)}}$
by Proposition \ref{Pr:dd-win-semiperf}, 
which gives a $\u D^{(n)}$-window $\u M^{(n)}=(h^{(n)})^*\u M{}_{R^{(n)}}$. 
These $\u M^{(n)}$ for $1\le n\le 3$ define a descent datum on $\u M^{(1)}$, 
which thus descends to a $\u D$-window $\u M$ 
by faithfully flat descent applied to the components of $\u M^{(1)}$.

For the same $\u\MMM$,
if we take $A(m)\to R$ with $m\ge 0$ instead of $A\to R$,
the same construction gives a $\u D(m)$-window $\u M(m)$.
This will involve obvious frames $\u D(m)^{(n)}$,
which we do not make explicit.
For each homomorphism $\u D(m)\to\u D(m')$ given by the cosimplicial
structure, $\u M(m')$ is the base change of $\u M(m)$ in a compatible way.
This gives a HPD stratification $\varepsilon$ on $\u M=\u M(0)$,
and we can define the functor $G$ by
\[
G(\u\MMM)=(\u M,\varepsilon).
\]

Assume that $F(\u M,\varepsilon)=\u\MMM$ and $G(\u\MMM)=(\u N,\varepsilon)$.
Then 
\chang{$\u M_{R^{(n)}}$} 
used in the construction of $G(\u\MMM)$
\chang{is} 
the base change of $\u M$ \chang{along} the frame homomorphisms 
$\u D\to\u A{}_{\cris}(R^{(n)})$ induced by the homomorphisms
$A\to \tilde A^{(n)}$ defined by the choice of one component;
this is independent of the choice 
using the HPD stratification $\varepsilon$ as explained above.
It follows that $\u N\cong\u M$, and one verifies that
this isomorphism preserves the stratifications, 
thus $G\circ F\cong\id$.
It remains to show that $G$ is fully faithful.
Assume that $G(\u\MMM)=(\u M,\varepsilon)$ and $G(\u\NNN)=(\u N,\varepsilon)$.
We have a commutative diagram with exact rows:
\[
\xymatrix@M+0.2em{
0 \ar[r] &
\Hom(\u\MMM,\u\NNN) \ar[r] \ar[d]_G &
\Hom(\u M{}_{R^{(1)}},\u N{}_{R^{(1)}}) \ar[r] \ar[d]^\cong &
\Hom(\u M{}_{R^{(2)}},\u N{}_{R^{(2)}}) \ar[d] \\
0 \ar[r] &
\Hom(\u M,\u N) \ar[r] &
\Hom(\u M^{(1)},\u N^{(1)}) \ar[r] & 
\Hom(\u M^{(2)},\u N^{(2)})
}
\]
Clearly the arrow labelled $G$ is injective.
We have to show that an $f\in\Hom(\u M,\u N)$ 
which commutes with the HPD stratifications 
maps to zero in $\Hom(\u M{}_{R^{(2)}},\u N{}_{R^{(2)}})$.
This condition depends only on the module homomorphism $f:M\to N$.
Since the proposition holds for finite locally free modules
in place of windows, the assertion follows.
\end{proof}

\begin{Cor}
\label{Co:DD-Win}
Let $R$ be a ring in which $p$ is nilpotent.
For a presentation $A\to R$ with a Frobenius lift $\sigma:A\to A$
and the associated frame $\u D$ there is an equivalence
$\DDcat(\Spec R)\cong\Win(\u D)^\nabla$
between divided Dieudonn\'e crystals over $X$
and windows over $\u D$ with a connection.
\end{Cor}

\begin{proof}
Propositions \ref{Pr:nabla-HPD} and \ref{Pr:DD-HPD}.
\end{proof}

Using Proposition \ref{Pr:DD-functor}, 
Theorem \ref{Th:Ffin-Fnil} and Corollary \ref{Co:Ffin-Fnil} it follows:

\begin{Cor}
\label{Co:BT-Win}
There is a functor $\BT(\Spec R)\to\Win(\u D)^\nabla$,
which is an equivalence if $R/p$ is $F$-finite and $F$-nilpotent
and if $p\ge 3$ or $pR=0$. \qed
\end{Cor}

\section{Divided Dieudonn\'e crystals and displays}
\label{Se:displays}

Finally we mention a link between divided Dieudonn\'e crystals and
the displays of \cite{Zink:Display} (called $3n$-displays in loc.cit.). Let $X$ be a scheme on which $p$ is locally nilpotent. 
As earlier we consider $X_\nameoftop$, 
the category of $\nameoftop$-sheaves on the category of $X$-schemes. 
There is a frame 
\[
\u W(\OOO_X)=(W(\OOO_X),I(\OOO_X),\OOO_X,\sigma,\sigma_1)
\]
in $X_\nameoftop$ whose value in $\Spec R\to X$ is the Witt frame $\u W(R)$
of Example \ref{Ex:Witt-frame}. A window over $\u W(\OOO_X)$ will be called a display over $X$, and $\Disp(X)$ denotes the category of displays over $X$.
A display over $X=\Spec R$ is a display over $R$ in the usual sense, i.e.\ a window over $\u W(R)$, by faithfully flat descent of displays; see \cite[Theorem 37]{Zink:Display}. 
By \cite{Lau:Smoothness} there is a contravariant functor 
\[
\Phi_X:\BT(X)\to\Disp(X),
\]
which restricts to an equivalence between formal (resp.\ unipotent) $p$-divisible groups and $F$-nilpotent (resp.\ $V$-nilpotent) displays.
In fact, the functor considered in loc.cit.\ is covariant, 
but we get a contravariant functor by taking the dual on either side.

\begin{Lemma}
There is a natural frame homomorphism $\pi:\u\OOO_X^{\cris}\to\u W(\OOO_X)$
over the identity of $\OOO_X$.
\end{Lemma}

\begin{proof}
It suffices to define for each $\Spec R\to X$ with semiperfect $R$ a frame homomorphism $\pi:\u A{}_{\cris}(R)\to\u W(R)$. The ring homomorphism $\pi$ is the unique homomorphism $A_{\cris}(R)\to W(R)$ of \chang{$p$-complete} PD thickenings of $R$. One verifies that $\pi$ is a frame homomorphism, using that
for $a\in J=\Ker(R^{\flat}\to R)$, the elements $[a]^{[n]}$ lie in 
the kernel of $A_{\cris}(R)\to W(R)$, and that 
$\sigma_1([a]^{[n]})=\frac{(pn!)}{n!\cdot p}[a]^{[pn]}$.
\end{proof}

We get a sequence of functors 
\[
\BT(R)\xrightarrow{\DDfun_X}\DDcat(X)\xrightarrow{\pi^*}\Disp(X).
\]
The composition coincides with the functor $\Phi_X$ by the uniqueness properties of $\Phi_X$ as in \cite[Proposition 2.1]{Lau:Smoothness}.

A window $\u M$ over a frame $\u A$ (or over a frame $\u\AAA$ in a topos)
is called $F$-nilpotent if the endomorphism $\Phi:M\to M$ 
is nilpotent on $M/pM$ (or 
\chang{nilpotent locally}).
We denote by $\Win(\u A)_{\nil}\subseteq\Win(A)$ 
the full subcategory of all $F$-nilpotent windows.

\begin{Prop}
\label{Pr:deform-win-nil}
A frame homomorphism $\alpha:\u A'\to\u A$ such that $A'\to A$ is surjective and $R'\to R$ ist bijective induces an equivalence of $F$-nilpotent windows.
\end{Prop}

\begin{proof}
We define $\u B{}_n$ as in the proof of Proposition \ref{Pr:deform-win}.
The frame homomorphism $\u B{}_n\to\u A$ with kernel $\bar N$ induces an equivalence of $F$-nilpotent windows by \cite[Theorem 10.3]{Lau:Frames}, applied to the sequence of ideals $\bar N\supseteq p\bar N\supseteq\ldots\supseteq p^n\bar N=0$, and the assertion follows by
\cite[Lemma 2.11]{Lau:Frames}.
\end{proof}

The above functors restrict to functors
\chang{\[
\BT(R)_{\inf}\xrightarrow{\DDfun_X}\DDcat(X)_{\nil}\xrightarrow{\pi^*}\Disp(X)_{\nil}
\]}
between the categories of formal $p$-divisible groups and $F$-nilpotent
divided Dieudonn\'e crystals or displays. Here the composite functor
is an equivalence by \cite[Theorem 5.1]{Lau:Smoothness}.

\begin{Prop}
The functor $\DDcat(X)_{\nil}\xrightarrow{\pi^*}\Disp(X)_{\nil}$
is an equivalence.
\end{Prop}

\begin{proof}
By $\nameoftop$-descent it suffices to show that for \chang{every} semiperfect ring $R$ the frame homomorphism $\pi:A_{\cris}(R)\to W(R)$ induces an equivalence of $F$-nilpotent windows. This follows from Proposition \ref{Pr:deform-win-nil}.
\end{proof}

As a consequence, 
\chang{$\DDfun_X:\BT(R)_{\inf}\to\DDcat(X)_{\nil}$} 
is an equivalence as well. This is more general than the restriction of Theorem \ref{Th:Ffin-Fnil} to formal $p$-divisible groups because there are no finiteness conditions on $X$.

\end{document}